\documentclass[11pt]{amsart}
\usepackage{amssymb,amsmath,latexsym,enumerate, dsfont}
\usepackage{graphicx,verbatim,enumerate,%
ifpdf,mdwlist}
\usepackage{mathrsfs}
\usepackage{mathabx}
\usepackage{mathtools}
\usepackage{color}
\ifpdf
\usepackage{hyperref}
\else
\fi
\usepackage{graphicx}
\usepackage{cancel}
\usepackage{float}
\renewcommand{\phi}{\varphi}

\newtheorem{theorem}{Theorem}[section]
\newtheorem{fact}[theorem]{Fact}
\newtheorem{lemma}[theorem]{Lemma}
\newtheorem{sublemma}{Sublemma}

\newtheorem{corollary}[theorem]{Corollary}

\newtheorem{question}[theorem]{Question}
\newtheorem*{lemma*}{Lemma}
\newtheorem{defi}[theorem]{Definition}
\newenvironment{definition}{\begin{defi} \rm}{ \end{defi}}
\newtheorem{exa}[theorem]{Example}

\newtheorem{rem}[theorem]{Remark}
\newenvironment{remark}{\begin{rem} \rm}{ \end{rem}}

\DeclareMathOperator{\dom}{domain}
\DeclareMathOperator{\up}{Up}

\newcommand{\mathand}{\; \&  \;}

\newcommand{\smf}{\smallfrown}

\newcommand{\aaa}{\mathcal{A}}

\newcommand{\imp}{\rightarrow}

\newcommand{\meet}{\wedge}
\newcommand{\andd}{\wedge}

\newcommand{\Th}{\operatorname{Th}}
\newcommand{\IPC}{\operatorname{IPC}}
\newcommand{\CPC}{\operatorname{CPC}}
\newcommand{\JAN}{\operatorname{JAN}}
\newcommand{\Dym}{\mathscr{M}_{D}}

\newcommand{\Usl}{\mathbb{U}\textrm{sl}}
\DeclareMathOperator{\graph}{graph}

\newcommand{\cat}[2]{#1\widehat{\phantom{\alpha}}\!#2}

\newcommand{\MP}[1]{\ensuremath{\mathcal{#1}}}
\newcommand{\QA}[1]{(\forall{#1})\,}
\newcommand{\QE}[1]{(\exists{#1})\,}

\AtBeginDocument{%
   \def\MR#1{}
}

\title[Intuitionism and computing with partial information]{Intuitionism
and computing with partial information}

\author[H.~Ganchev]{Hristo Ganchev}
\address{Department of Mathematical Logic, Faculty of Mathematics and Computer Science, Sofia University, Sofia}
\email{\href{mailto:ganchev@fmi.uni-sofia.bg}{ganchev@fmi.uni-sofia.bg}}

\author[P.~Shafer]{Paul Shafer}
\address{School of Mathematics\\
University of Leeds\\
Leeds, LS2 9JT, United Kingdom}
\email{\href{mailto:p.e.shafer@leeds.ac.uk}{p.e.shafer@leeds.ac.uk}}
\thanks{The second author was supported by EPSRC Overseas Travel Grant EP/R006458/1.}

\author[T.~Slaman]{Theodore A. Slaman}
\address{Department of Mathematics, University of California, Berkeley, Berkeley, CA 94720, USA}
\email{\href{mailto:slaman@math.berkeley.edu}{slaman@math.berkeley.edu}}

\author[A.~Sorbi]{Andrea Sorbi}
\address{Dipartimento di Ingegneria Informatica e Scienze Matematiche\\
Universit\`a Degli Studi di Siena\\
I-53100 Siena, Italy}
\email{\href{mailto:andrea.sorbi@unisi.it}{andrea.sorbi@unisi.it}}
\thanks{The fourth author is a member of INDAM-GNSAGA}

\author[M.~Soskova]{Mariya I. Soskova}
\address{Department of Mathematics, University of Wisconsin\\
Madison, USA}
\email{\href{mailto:msoskova@math.wisc.edu}{msoskova@math.wisc.edu}}
\thanks{The fifth author was partially supported by NSF grant DMS-2451099.}



\begin{document}

\begin{abstract}
There exist initial segments of both the Dyment lattice and the 
Dyment–Muchnik lattice that yield Brouwer algebras modeling exactly the 
intuitionistic propositional calculus. For the Dyment–Muchnik lattice, this 
result is obtained by constructing a splitting class of enumeration 
degrees. In contrast, the full Dyment lattice and the full Dyment–Muchnik 
lattice model the intuitionistic propositional calculus plus the weak law 
of excluded middle. We also observe that certain naturally definable 
classes of enumeration degrees, which are downwards closed under 
enumeration reducibility, fail to form splitting classes.  
\end{abstract}

\maketitle

\section{Introduction}
Mass problems were introduced by Medvedev in \cite{Medvedev} to formalize 
Kolmogorov’s proposal of a calculus of problems as a semantics for 
constructive logics. The central idea is to identify “problems’’ with mass 
problems, where a \emph{mass problem} is a set of total functions from 
$\omega$ to $\omega$ (with $\omega$ denoting the set of natural numbers). The 
elements of a mass problem are viewed as its solutions. A mass problem is 
\emph{solvable} if it contains a computable element, i.e.\ if we can compute 
a solution to the corresponding problem. 

Given two mass problems $\MP{A}$ and $\MP{B}$, we say that $\MP{A}$ is 
\emph{Medvedev reducible} to $\MP{B}$ (written $\MP{A} \leq \MP{B}$) if there 
exists an oracle Turing machine that computes an element of $\MP{A}$ whenever 
it is supplied with an element of $\MP{B}$ as an oracle. Intuitively, this 
implies that we can solve $\MP{A}$ if we can solve $\MP{B}$. Under this 
reducibility, the solvable mass problems are the “easiest possible,’’ since 
they reduce to every other mass problem. 

Following Kolmogorov’s idea of interpreting logical connectives as operations 
on problems, observe that we can solve $\MP{A}$ \emph{or} $\MP{B}$ if and 
only if we can solve 
\[
  \MP{A} \wedge \MP{B}
    = \cat{0}{\MP{A}} \cup \cat{1}{\MP{B}},
\]
where $\cat{i}{\MP{C}} = \{ \cat{i}{f} : f \in \MP{C} \}$ and $\cat{i}{f}$ is 
the function defined by $\cat{i}{f}(0) = i$ and $\cat{i}{f}(x+1) = f(x)$. 
Likewise, we can solve $\MP{A}$ \emph{and} $\MP{B}$ if and only if we can 
solve 
\[
  \MP{A} \oplus \MP{B} = \{ f \oplus g : f \in \MP{A},\ g \in \MP{B} \},
\]
where $f \oplus g$ is defined by $f \oplus g(2x) = f(x)$ and $f \oplus 
g(2x+1) = g(x)$. One can further define 
\[
  \MP{A} \to \MP{B}
    = \left\{ \cat{e}{f} : (\forall g \in \MP{A})\,
        [\, \Phi_e(f \oplus g) \downarrow \in \MP{B} \,] \right\},
\]
where $\Phi_e$ is the $e$-th oracle Turing machine, and $\Phi_e(f \oplus g) 
\downarrow$ means that $\Phi_e(f \oplus g)$ is a total function. If this mass 
problem is solvable, then informally the statement “if $\MP{A}$ is solvable, 
then $\MP{B}$ is solvable’’ holds. Finally, one interprets the connective 
$\neg$ by 
\[
  \neg \MP{A} = \MP{A} \to \emptyset .
\]

Muchnik later proposed in \cite{Muchnik:Lattice} a weakening of Medvedev 
reducibility, called \emph{weak reducibility}, defined by 
\[
  \MP{A} \leq_w \MP{B}
    \quad\text{iff}\quad
  (\forall g \in \MP{B})(\exists f \in \MP{A})[ f \leq_T g ],
\]
where $\leq_T$ denotes Turing reducibility. This weaker reducibility induces 
the same interpretation of the propositional connectives $\land$ and $\lor$ 
on mass problems and the same notion of solvability. However, the Muchnik 
interpretation of $\to$ is 
\[
\MP{A} \to \MP{B} = \left\{f : (\forall g \in \MP{A})(\exists h 
\in \MP{B})\, [h \leq_T f \oplus g] \right\},
\]
which need not be equivalent to the Medvedev interpretation. Under these 
interpretations, a propositional formula is \emph{valid} in either the 
Medvedev or the Muchnik calculus of problems if, for every substitution of 
mass problems for propositional atoms, the resulting mass problem is 
solvable. It turns out that the valid formulas are precisely the theorems of 
the intermediate logic obtained by adding to the intuitionistic propositional 
calculus the weak law of excluded middle, i.e.\ the axiom scheme $\neg \alpha 
\lor \neg\neg \alpha$ (\cite{Medvedev, Sorbi:Embedding}). This logic is known 
as the \emph{Jankov logic}, or the \emph{logic of the weak law of the 
excluded middle}. We will call $JAN$ this logic after Jankov. 

For many years it was open whether there exists a mass problem $\MP{C}$ such 
that the mass problems reducible to $\MP{C}$ model exactly the intuitionistic 
propositional logic. For the Medvedev lattice this question was answered 
positively in \cite{Skvortsova:intuitionism}, and for the Muchnik lattice in 
\cite{Sorbi-Terwijn:factors}. Both results were later improved by Kuyper 
\cite{Kuyper-Muchnik, Kuyper-Medvedev}, who explicitly exhibited initial 
segments arising from “natural’’ (or at least “more natural’’) mass problems. 
By contrast, the proofs of Skvortsova and of Sorbi--Terwijn rely at some 
point on the classical theorem of Lachlan and Lebeuf \cite{Lachlan-Lebeuf}, 
which states that every countable upper semilattice with a least element is 
isomorphic to an initial segment of the Turing degrees. 

In both Medvedev’s and Muchnik’s approaches the emphasis is on total 
functions as solutions to problems and on Turing computability as the means 
of reducing mass problems to one another. In this paper we instead adopt the 
perspective of computability relative to \emph{partial} information. The most 
general and comprehensive framework for relative computability of partial 
functions is the model proposed by Myhill \cite{Myhill:Partial}, according to 
which a partial function $\phi$ is computable relative to a partial function 
$\psi$ (written $\phi \leq_e \psi$) if there exists an effective enumeration 
procedure that enumerates the graph of $\phi$ from \emph{any} enumeration of 
the graph of $\psi$. The precise meaning of “effective enumeration 
procedure’’ will be given later, when we introduce enumeration operators. 

(An alternative and suggestive formulation of relative computability of 
partial functions---see, e.g., \cite{McEvoy:Dissertation}---states that 
$\phi$ is computable by a nondeterministic partial-function-oracle Turing 
machine using $\psi$ as oracle; such a machine always loops when it queries 
an undefined value of the oracle.) 

Within this framework it is natural---following Dyment 
\cite{Dyment:Someproperties}---to identify informal problems with sets of 
partial functions from $\omega$ to $\omega$. Reducibility of mass problems 
then refers to relative computability of partial functions, and the 
“easiest’’ problems (those reducible to every other) are precisely those 
containing partial computable functions. 

\subsection{The results} 

In Section~\ref{sec:Dyment} we show that there exists an initial segment of 
the Dyment lattice such that the propositional formulas which are valid in 
this initial segment are exactly the theorems of the intuitionistic 
propositional calculus. For this investigation, most of the machinery and of 
the ingredients used in the proof will turn out to be almost straightforward 
modifications of the ones employed for the Medvedev lattice in 
\cite{Kuyper-Medvedev}, \cite{Skvortsova:intuitionism}, \ 
\cite{Sorbi:Embedding}. In addition to some intended simplifications, the 
most relevant exceptions are perhaps the proofs of Sublemma~\ref{sublem:1} 
and Sublemma~\ref{sublem:2} to Lemma~\ref{lem:can1}, which seem to be the 
only cases where the Dyment lattice demands a proof of its own. In 
Section~\ref{sec:Dyment-Muchnik} we show that there exists an initial segment 
of the Dyment-Muchnik lattice modeling again exactly the intuitionistic 
propositional calculus. For this result, the known proof for the Muchnik 
lattice in \cite{Sorbi-Terwijn:factors} is of no help, since it relies on the 
Lachlan and Lebeuf result on the initial segments of the Turing degrees, 
which does not hold in the enumeration degrees. We do this by exhibiting a 
splitting class in the enumeration degrees, not an easy task in fact, since 
the known examples of splitting classes in the Turing degrees do not seem to 
be easily exportable to the context of enumeration reducibility, as 
exemplified in Subsection~\ref{ssct:c}. It follows from these results 
(Section~\ref{sec:total}) that both the full Dyment lattice and the full 
Dyment-Muchnik lattice model the Jankov logic $\JAN$. 

\section{Background and notations}

In this section we fix some of the most used notions and notations.

\subsection{Computability theory}\label{ssct:comp}

The reader is referred to any standard textbook on computability theory (see, 
for instance, \cite{Cooper:Book, Rogers:Book, Soare:Book}) for the basic 
background. Recall that an \emph{enumeration operator} (or 
\emph{$e$-operator}) is a mapping $\Phi: 2^{\omega} \to 2^{\omega}$ for which 
there exists a computably enumerable (c.e.) set $W$ (said to \emph{define} 
$\Phi$) such that, for $A \subseteq \omega$, 
\[
\Phi(A)=\{x: (\exists u)[\langle x, u\rangle \in W 
                        \mathand D_{u} \subseteq A]\},
\]
where $D_{u}$ is the finite set with canonical index $u$. (In the remainder 
of the paper we will often identify finite sets with their canonical indices, 
so we may write $\langle x, D \rangle \in W$ instead of $\langle x, u \rangle 
\in W$, where $D = D_{u}$.) 

Given $A, B \subseteq \omega$, we say that $A$ is \emph{enumeration 
reducible} (or \emph{$e$-reducible}) \emph{to} $B$ (notation: $A \leq_e B$) 
if there exists an enumeration operator $\Phi$ such that $A = \Phi(B)$. This 
reducibility induces in the usual way a degree structure, namely the 
structure $\mathcal{D}_e^{\mathrm{set}}$ of \emph{$e$-degrees}. One may refer 
to the \emph{standard indexing} $(\Phi_e : e \in \omega)$ of the enumeration 
operators, where $\Phi_e$ is the operator defined by the c.e.\ set $W_e$. 

Let $\mathcal{P}$ denote the set of all partial functions from $\omega$ to 
$\omega$. For $\psi \in \mathcal{P}$, an \emph{enumeration of $\psi$} means 
an enumeration of $\operatorname{graph}(\psi)$. The definition of 
$e$-reducibility transfers in the obvious way from sets to partial functions 
by declaring that for partial functions $\phi$ and $\psi$, we have $\phi 
\leq_e \psi$ if $\operatorname{graph}(\phi) \leq_e 
\operatorname{graph}(\psi)$. The \emph{partial degree} of $\phi$ is the 
equivalence class of $\phi$ under the equivalence relation $\equiv_e$ induced 
by $\leq_e$ on $\mathcal{P}$. The structure of the partial degrees will be 
denoted by $\mathcal{D}_e^{\mathrm{pf}}$. 

If $\phi, \psi \in \mathcal{P}$ and $x$ is a number, then $\phi(x)=\psi(x)$ 
means that either both $\phi(x)$ and $\psi(x)$ are undefined, or they are 
both defined and yield the same value. 

Any enumeration operator $\Phi$ induces a partial mapping $\Phi: \mathcal{P} 
\to \mathcal{P}$ as follows: for $\psi \in \mathcal{P}$, we say that 
$\Phi(\psi)\downarrow$ (or $\psi \in \dom(\Phi)$) if 
$\Phi(\operatorname{graph}(\psi))$ is single-valued, i.e.\ there exists $\chi 
\in \mathcal{P}$ such that $\Phi(\operatorname{graph}(\psi)) = 
\operatorname{graph}(\chi)$. In this case we write $\Phi(\psi)\downarrow = 
\chi$. If $\Phi(\psi)\downarrow$ and $x,y \in \omega$, then 
$\Phi(\psi)(x)\downarrow = y$ means that $\langle x, y \rangle \in 
\Phi(\psi)$, and $\Phi(\psi)(x)\downarrow$ means $(\exists z)\,[\langle x, z 
\rangle \in \Phi(\psi)]$. 

\subsection{Lattices and Brouwer algebras} 

Our main reference for lattice theory is \cite{BalbesDwinger}. A 
\emph{Brouwer algebra} is a lattice $\mathscr{L}=\langle L, \lor, \wedge, 1, 
\leq \rangle$ (here, $\leq$ denotes the partial ordering of the lattice, 
defined by $a\leq b$ iff $a\wedge b=a$ iff $a\lor b=b$; the symbol $\geq$ 
denotes the dual relation of $\leq$) with a greatest element such that for 
every $a,b \in L$, the set $\{c: a \lor c \geq b\}$ has a least element 
(denoted by $a \rightarrow b$). An equivalent formulation is that the type of 
$\mathscr{L}$ can be enriched with a binary operation $a \rightarrow b$, such 
that for all  $a,b,c \in L$, $a \lor c \geq b \Leftrightarrow a\rightarrow b 
\le c$. The Brouwer algebras form an equational class. It is a routine 
exercise to verify that every Brouwer algebra has a least element as well, 
and that every Brouwer algebra is a distributive lattice. Hence we may refer 
to a Brouwer algebra as an algebra $\mathscr{L}=\langle L, \lor, \wedge, 
\rightarrow, 0, 1\rangle$ where $\langle L, \lor, \wedge,  0, 1\rangle$ is a 
bounded distributive lattice, and $a\rightarrow b=\min \{c: a \lor c \geq 
b\}$. We will always assume that whether $\mathscr{L}$ is a Brouwer algebra 
or simply a lattice, it is nontrivial, i.e., $0\ne 1$. Every finite 
distributive lattice is easily seen to be a Brouwer algebra. Finally, if 
$\mathscr{L}=\langle L, \lor, \wedge, \rightarrow, 0, 1\rangle$ is a Brouwer 
algebra and $a\ne 0$ then the interval $[0,a]_{\mathscr{L}}= \langle \{b: b 
\leq a\}, \lor, \wedge, \rightarrow, 0, a\rangle$ is a Brouwer algebra as 
well, with greatest element $a$, and with the operations obtained by 
restricting those of $\mathscr{L}$ to the interval. 
 
\subsection{Propositional logics} 

Brouwer algebras can be used for a semantics of propositional logics. A 
propositional formula $\alpha$ is an \emph{identity} of a Brouwer algebra 
$\mathscr{L}$ if every substitution of elements of $\mathscr{L}$ for the 
propositional atoms of $\alpha$ always yields the least element $0$ of the 
algebra, when the propositional connectives $\lor$, $\wedge$, and 
$\rightarrow$ are interpreted by the algebraic operations $\wedge$, $\lor$, 
and $\rightarrow$, respectively (note that the connective $\lor$ is 
interpreted as the meet operation and $\wedge$ as the join operation). 
Negation $\neg$ is interpreted by the operation $\neg a = a \rightarrow 1$, 
by which we can enrich the type of the Brouwer algebra. The \emph{theory} 
$\Th(\mathscr{L})$ of a Brouwer algebra $\mathscr{L}$ is the set of its 
identities. It is easy to see that $\IPC \subseteq \Th(\mathscr{L}) \subseteq 
\CPC$ for every Brouwer algebra $\MP{L}$, where $\IPC$ denotes the set of 
theorems of the intuitionistic propositional calculus, and $\CPC$ denotes the 
set of classically valid propositional formulas. Straightforward 
verifications show the following. 

\begin{fact}\label{fact:1}
Let $f: \mathscr{L}_1 \longrightarrow \mathscr{L}_2$ be a homomorphism of 
Brouwer algebras. If $f$ is one-one then $\Th{(\mathscr{L}_2)}\subseteq 
\Th{(\mathscr{L}_1 )}$;  if $f$ is onto, then  $\Th{(\mathscr{L}_1)}\subseteq 
\Th{(\mathscr{L}_2)}$. Moreover,  for every Brouwer algebra $\mathscr{L}$ and 
elements $x,y, z \in \mathscr{L}$ such that $x<y$ and $y= z \lor x$, then 
$\Th([0,z]_{\mathscr{L}})\subseteq \Th([x,y]_{\mathscr{L}})$. 
\end{fact}

\begin{proof}
The statements concerning the homomorphism $f: \mathscr{L}_1 \longrightarrow 
\mathscr{L}_2$ follow from straightforward verifications. For the remaining 
claim, note (see \cite[Lemma~4]{Skvortsova:intuitionism}) that under the 
given assumptions the mapping $u \mapsto x \lor u$ is a Brouwer algebra 
homomorphism from $[0,z]_{\mathscr{L}}$ onto $[x,y]_{\mathscr{L}}$. 
\end{proof}

For further details on the algebraic semantics of propositional logic the 
reader is referred to the classic but still valuable text by Rasiowa and 
Sikorski \cite{Rasiowa-Sikorski:Book}. 

\subsection{Notations}\label{ssct:notations}

The following notation will be used throughout the paper. The symbols 
$\Usl_{\rightarrow}$, $\mathbb{L}_d$, and $\mathbb{B}$ denote, respectively, 
the categories of implicative upper semilattices, distributive lattices, and 
Brouwer algebras: an upper semilattice (with least upper bound operation 
$\lor$, and with partial ordering relation $\leq$) is \emph{implicative} if 
it can be equipped with a binary operation $\rightarrow$ satisfying that $a 
\rightarrow b=\textrm{least }\{c: b \leq a\lor c\}$. The symbol 
$Fr^{L_d}_{\Usl_{\rightarrow}}(\mathscr{U})$ denotes the free distributive 
lattice over the implicative upper semilattice $\mathscr{U}$. Given two 
objects $\mathscr{L}_1$ and $\mathscr{L}_2$ in any of these categories, say 
$\mathbb{C}$, we write $\mathscr{L}_1 \overset{\mathbb{C}}{\simeq} 
\mathscr{L}_2$ (or simply $\mathscr{L}_1 \simeq \mathscr{L}_2$ where $\simeq$ 
denotes order-theoretic isomorphism, since in all of these categories the 
objects are partially ordered structures and the isomorphisms coincide with 
the order-theoretic isomorphisms) to indicate that the two objects are 
isomorphic. We write $f : \mathscr{L}_1 \xhookrightarrow{\mathbb{C}} 
\mathscr{L}_2$ to indicate that $f$ is a monomorphism in $\mathbb{C}$ from 
$\mathscr{L}_1$ to $\mathscr{L}_2$.   

\section{The Dyment lattice and the Dyment-Muchnik lattice} 

A \emph{Dyment mass problem} (or a \emph{$D$-mass problem}) $\MP{A}$ is any 
subset of the collection $\MP{P}$ of all partial functions from $\omega$ to 
$\omega$.  Given $D$-mass problems $\MP{A}, \MP{B}$, we say that $\MP{A}$ is 
\emph{Dyment-reducible} (or, simply, \emph{$D$-reducible}) to  $\MP{B}$ 
(notation $\MP{A} \leq_D \MP{B}$), if there exists an enumeration operator 
$\Phi$ such that $\mathcal{B}\subseteq \dom(\Phi)$, and $\Phi(\psi) \in 
\mathcal{A}$, for every $\psi \in \MP{B}$. The \emph{Dyment degree} (or 
\emph{$D$-degree}) of $\MP{A}$, denoted by $[\MP{A}]_{D}$, is the equivalence 
class of $\MP{A}$ under the equivalence relation $\equiv_{D}$ induced by the 
preordering relation $\le_{D}$. The set of Dyment degrees is denoted by 
$\textbf{M}_{D}$. 

Similar to Muchnik's weakening of Medvedev's reducibility, one can define the 
weaker non uniform version of $D$-reducibility on $D$-mass  problems, for 
which $\MP{A}$ is \emph{Dyment-Muchnik reducible to $\MP{B}$} (notation: 
$\MP{A}\leq_{D,w} \MP{B}$) if $(\forall \psi \in \MP{B})(\exists \phi \in 
\MP{A})[\phi \leq_{e} \psi]$.  The \emph{Dyment-Muchnik degree} of $\MP{A}$, 
denoted by $[\MP{A}]_{D,w}$, is the equivalence class of $\MP{A}$ under the 
equivalence $\equiv_{D,w}$ induced by the preordering relation $\le_{D,w}$. 
The set of Dyment-Muchnik degrees is denoted by $\textbf{M}_{D,w}$. 

Roughly Dyment-Muchnik reducibility is to Dyment reducibility as Muchnik 
reducibility is to Medvedev reducibility. 

We now extend to partial functions and $D$-mass problems the familiar 
operations used by Medvedev on total functions and mass problems of total 
functions. Given partial functions $\phi, \psi \in \mathcal{P}$ define $\phi 
\oplus \psi$ to be the partial function 
\[
\phi \oplus \psi(x)=
\begin{cases}
 \phi(y)     & \text{if $x=2y$}, \\
 \psi(y)   & \text{if $x=2y+1$}.
\end{cases}
\]
If $i\in \omega$ and $\phi\in \mathcal{P}$, define $\cat{i}{\phi}$ to be the
partial function
\[
\cat{i}{\phi}(x)=
\begin{cases}
i  & \text{if $x=0$ }, \\
\phi(x-1)  & \text{if $x>0$},
\end{cases}
\]
and if $\MP{A}$ is a $D$-mass  problem then denote $\cat{i}{\MP{A}}=
\{\cat{i}{\phi}: \phi \in \MP{A}\}$. Finally, let
\begin{align*}
&\MP{A} \oplus  \MP{B}=\{\phi \oplus \psi: \phi\in \MP{A}, \psi \in \MP{B}\},\\
&\MP{A} \wedge \MP{B}=\cat{0}{\MP{A}} \cup \cat{1}{\MP{B}},\\
&\MP{A} \imp \MP{B}=\{\cat{e}{\phi}:
        (\forall \psi\in \MP{A})[\Phi_{e}(\phi\oplus \psi)\downarrow \in
        \MP{B}]\}.
\end{align*}
The Dyment equivalence relation $\equiv_D$ on $D$-mass  problems is a 
congruence with respect to the above introduced operations on $D$-mass  
problems (\cite{Dyment:Someproperties}). Hence, similarly to the case of 
Medvedev reducibility, these operations induce corresponding operations (for 
which we use the same symbols) in the quotient structure. If $\mathbf{0}_{D}$ 
denotes the Dyment degree of the $D$-mass  problem containing all partial 
computable functions, and $\mathbf{1}_{D}$ denotes the Dyment degree of the 
empty $D$-mass  problem, then the following holds. 

\begin{theorem}
The algebra $\mathscr{M}_{D}= \langle \textbf{M}_{D}, \lor, \wedge, 
\rightarrow, \mathbf{0}_{D}, \mathbf{1}_{D}\rangle$ is a Brouwer algebra, 
with least element given by $\mathbf{0}_{D}$, and greatest element given by  
$\mathbf{1}_{D}$, and 
\begin{align*}
&[\MP{A}]_{D} \lor [\MP{B}]_{D}=[\MP{A}\oplus  \MP{B}]_{D},\\
&[\MP{A}]_{D} \wedge [\MP{B}]_{D}=[\MP{A} \wedge \MP{B}]_{D},\\
&[\MP{A}]_{D} \imp [\MP{B}]_{D}=[\MP{A} \imp \MP{B}]_{D}.
\end{align*}
The partial ordering relation $\le_{D}$ on $\mathscr{M}_{D}$ (for which we 
use the same symbol as the preordering relation from which it originates) is 
given by $[\MP{A}]_{D}\le_{D} [\MP{B}]_{D}$ if $\MP{A}\leq_{D} \MP{B}$. 
\end{theorem}

\begin{proof}
The proof is essentially the same as the proof that the Medvedev lattice is a 
Brouwer algebra, see \cite{Medvedev} or \cite[\S~13.7]{Rogers:Book}. 
\end{proof}

\begin{definition}
$\mathscr{M}_{D}$ is called the \emph{Dyment lattice}. 
\end{definition}

The Dyment-Muchnik equivalence $\equiv_{D,w}$ on $D$-mass problems is a 
congruence with respect to the above introduced operations on $D$-mass  
problems (\cite{Dyment:Someproperties}), except with $\to$ now interpreted as
\[
\MP{A} \to \MP{B} = 
\{\chi: \QA{\phi \in 
\mathcal{A}}\QE{\psi\in \mathcal{B}}[\psi \leq_e \chi \oplus \phi]\}.
\]
Similarly to the case of Muchnik reducibility, these operations induce 
corresponding operations (for which we use the same symbols) in the quotient 
structure. If $\mathbf{0}_{D,w}$ denotes the Dyment-Muchnik degree of the 
$D$-mass  problem containing all partial computable functions, and 
$\mathbf{1}_{D,w}$ denotes the Dyment-Muchnik degree of the empty 
$D$-problem, then, again, it is straightforward to see that the following 
holds. 

\begin{theorem}
The algebra $\mathscr{M}_{D,w}= \langle \textbf{M}_{D,w}, \lor, \wedge, 
\rightarrow, \mathbf{0}_{D,w}, \mathbf{1}_{D,w}\rangle$ is a Brouwer algebra, 
with least element given by $\mathbf{0}_{D,w}$, and greatest element given by  
$\mathbf{1}_{D,w}$, and 
\begin{align*}
&[\MP{A}]_{D,w} \lor [\MP{B}]_{D,w}=[\MP{A}\oplus  \MP{B}]_{D,w},\\
&[\MP{A}]_{D,w} \wedge [\MP{B}]_{D,w}=[\MP{A} \wedge \MP{B}]_{D,w}=[\MP{A} \cup \MP{B}]_{D,w},\\
&[\MP{A}]_{D,w} \imp [\MP{B}]_{D,w}=[\MP{A} \imp \MP{B}]_{D,w}
\end{align*}
The partial ordering relation $\le_{D,w}$ on $\mathscr{M}_{D,w}$ (for which 
we use the same symbol as the preordering relation from which it originates) 
is given by $[\MP{A}]_{D,w}\le_{D,w} [\MP{B}]_{D,w}$ if $\MP{A}\leq_{D,w} 
\MP{B}$. 
\end{theorem}

\begin{proof}
Again, the proof is essentially the same as the proof that the Muchnik 
lattice is a Brouwer algebra, see \cite{Muchnik:Lattice} or 
\cite{Sorbi:Someremarks}. 
\end{proof}

\begin{definition}
$\mathscr{M}_{D,w}$ is called the \emph{Dyment-Muchnik lattice}.
\end{definition}

We will use boldface capital letters as variables for Dyment degrees (or, 
when clear from context, Dyment–Muchnik degrees). For more information about 
the Dyment lattice and the Dyment-Muchnik lattice  see 
\cite{Dyment:Someproperties}, \cite{Sorbi:Someremarks}, \cite{Sorbi-Survey}. 

\section{The Dyment lattice and intermediate logics: 
An initial segment of the Dyment lattice modeling exactly 
$\IPC$}\label{sec:Dyment} 

We split this section into three subsections. The first subsection reviews 
some known algebraic facts regarding Brouwer algebras, and known logical 
ingredients relating Brouwer algebras to the intuitionistic propositional 
logic. The second subsection shows some useful computability-theoretic facts 
relative to the Dyment lattice. The third subsection contains the main 
theorem stating that there is an initial segment of the Dyment lattice 
modeling $\IPC$. 

\subsection{The algebraic ingredients}\label{sct:alg-ingred}
If $A$ is a subset of a preordered set $\mathscr{P}=\langle P, \leq\rangle$ 
then let 
\begin{align*}
A_{u\leq}&= \{y\in P: (\exists x\in A)[y \ge x]\},\\
A_{d\leq}&= \{y\in P: (\exists x \in A)[y \le x]\}.
\end{align*}
We call $A_{u\leq}$ the \emph{upwards $\leq$-closure of $A$ in 
$\mathscr{P}$}, and we call $A_{d\leq}$  the \emph{downwards~$\leq$-closure 
of $A$ in $\mathscr{P}$}. If $A=A_{u\leq}$ then we say that $A$ is 
\emph{upwards~$\le$-closed in $\mathscr{P}$}, and if $A=A_{d\leq}$ then we 
say that $A$ is \emph{downwards~$\le$-closed in $\mathscr{P}$}.  

If $\mathscr{P}$ is a preorder as above, then let $\up(\mathscr{P})=\langle 
\up(P), \supseteq \rangle $ be the poset with universe 
\[
\up(P)=\{X \subseteq P: \textrm{ $X$ is upwards~$\leq$-closed in 
$\mathscr{P}$}\},
\]
and partially ordered by reverse inclusion. Simple calculations show that 
$\up(\mathscr{P})$ is a Brouwer algebra, with $\cup$ giving the operation 
$\wedge$ of greatest lower bound, $\cap$ giving the operation $\lor$ of least 
upper bound, $0=U$, $1= \emptyset$, and finally, 
\begin{equation*}\label{eqn:compute-arrow}
X\rightarrow Y=\bigcup \{Z \in \up(P) : Y \supseteq Z \cap X\}.
\end{equation*} 

Starting from a pre-upper semilattice $\mathscr{U} = \langle U,+,\le \rangle$ 
(see next Remark) one can carry out a similar construction, and get a Brouwer 
algebra $\up(\mathscr{U})$. In this case we can also describe $\lor$ as 
$X\lor Y=\{x+y: x \in X,\, y \in Y\}$, and $\rightarrow$ as $X\rightarrow 
Y=\{z: \QA{x \in X}[z+x\in Y]\}$.  

\begin{remark}\label{rem:iso-upsets}
By a \emph{pre-upper semilattice} $\mathscr{U} = \langle U,+,\le \rangle$ we 
mean that $\leq$ is a preordering relation and the equivalence relation 
originated by $\leq$ is a congruence with respect to $+$, so that the 
quotient structure $\mathscr{U}_{/\equiv} =\langle U_{/\equiv}, +_{/\equiv}, 
\leq_{/\equiv} \rangle$ is a an upper semilattice. A \emph{pre-upper 
semilattice with a least element} is defined similarly. It is easy to see 
that if $\mathscr{P}$ is a preorder then $\up(\mathscr{P}) \simeq 
\up(\mathscr{P}_{/\equiv})$ via the isomorphism $\mathcal{A} \mapsto 
\{a_{/\equiv}: a \in \mathcal{A}\}$, and consequently if $\mathscr{U}$ is a 
pre-upper semilattice, then $\up(\mathscr{U}) \simeq 
\up(\mathscr{U}_{/\equiv})$. 
\end{remark} 

Recall from Subsection~\ref{ssct:comp} that $\mathscr{D}_e^{pf}$ is the upper 
semilattice of the partial degrees. Let $\mathscr{P}_e= \langle 
\mathcal{P},\oplus, \leq_e \rangle$ be the pre-upper semilattice of the 
partial functions under enumeration reducibility $\leq_e$. Then

\begin{theorem}\label{thm:iso-usl-muchnik}
We have
\[
\mathscr{M}_{D,w} \simeq \up(\mathscr{D}_e^{pf}) 
\simeq \up(\mathscr{P}_e).
\]
\end{theorem}

\begin{proof}
For every pair of $D$-mass problems $\MP{A}$ and $\MP{B}$, we have 
\[
\MP{A} \leq_{D,w} \MP{B} \Leftrightarrow \MP{B} \subseteq 
\MP{A}_{u\leq_e}\Leftrightarrow 
\MP{B}_{u\leq_e} \subseteq \MP{A}_{u\leq_e}. 
\]
Hence the mapping $[\mathcal{A}]_{D,w} \mapsto \MP{A}_{u\leq_e}$ provides an 
isomorphism between $\mathscr{M}_{D,w}$ and $\up(\mathscr{P}_e)$. 
\end{proof}

\begin{definition}\cite[Definition~3.1]{Kuyper-Medvedev}\label{def:alpha0}
Given  a downwards~$\leq$-closed set $A$ in a pre-upper semilattice 
$\mathscr{U}=\langle U, +, \leq \rangle$, by a \emph{strong upwards antichain 
in $A$} (also, abbreviated as \emph{strong $u$--antichain in $A$}) we mean 
any collection of elements $\{x_i: i \in I\} \subseteq A$ such that 
\[
(\forall i,j \in I)[i \ne j \Rightarrow x_i + x_j \notin A].
\]
\end{definition}

An obvious instance of this definition is when $\{x_i: i \in I\}$ is an 
antichain, and $A = \{x_i: i \in I\}_{d\leq}$.  

\begin{definition}\label{def:alpha}
If $\mathscr{U}$ and $A$ are as above, and $\{x_i: i \in I\}$ is a strong 
$u$--antichain  in $A$,  then for every $X \subseteq I$ define 
\[
\alpha(X)=A^c \cup \{x_i: i \in X\}_{u\leq},
\]
where $A^c=U \smallsetminus A$. 
\end{definition}

\begin{lemma}\label{lem:alphab}
Notice that $\alpha(X)$ is upwards $\leq$-closed. Moreover, $\alpha(I)=A^c 
\cup \{x_i: i \in I\}_{u\leq}$ and $\alpha(\emptyset)=A^c$. 
\end{lemma}

\begin{proof} 
Immediate.
\end{proof}

\begin{lemma}\cite[Theorem~3.3]{Kuyper-Medvedev}\label{lem:algebraic}
If $\mathscr{U}$, $A$, $\{x_i: i \in I\}$, $\alpha$ are as in 
Definitions~\ref{def:alpha0} and \ref{def:alpha}, and $\mathscr{P}(I)=\langle 
\mathcal{P}(I), \supseteq \rangle$ is the Boolean algebra of the power set of 
$I$ ordered by reverse inclusion, then $\alpha: \mathscr{P}(I) 
\xhookrightarrow{\Usl_{ \rightarrow}} [\alpha(I), 
\alpha(\emptyset)]_{\up(\mathscr{U})}$ is an embedding which preserves both the least and the greatest elements. 
\end{lemma}

\begin{proof}
The proof follows from easy calculations: the lazy reader can look at 
\cite[Theorem~3.3]{Kuyper-Medvedev}, where these calculations  refer  to the 
case when $\mathscr{U}$ is the pre-upper semilattice $\langle \omega^\omega, 
\oplus,  \leq_T \rangle$ of total functions under Turing reducibility. 
\end{proof}

Recall that an element $x$ of a lattice is \emph{meet-reducible} if there are $y, z > x$
such that $x = y \wedge z$.  Dually, $x$ is \emph{join-reducible}  if there are $y, z < x$
such that $x = y \vee z$.  Elements that are not meet-reducible are called \emph{meet-irreducible},
and elements that are not join-reducible are called \emph{join-irreducible}.

\begin{definition}\cite{Skvortsova:intuitionism}\label{def:canonical}
A subset $\mathscr{C}$ of a Brouwer algebra $\mathscr{B}$ is called 
\emph{canonical in $\mathscr{B}$} if 
\begin{enumerate}
\item if $a \in \mathscr{C}$ and $a,b \in \mathscr{B}$ then $a \rightarrow 
    b \wedge c= (a\rightarrow b) \wedge (a\rightarrow c)$; 
  \item the elements of $\mathscr{C}$ are meet-irreducible; 
  \item $\mathscr{C}$ is a sub-implicative upper semilattice of 
      $\mathscr{B}$. 
\end{enumerate}
\end{definition}

\begin{lemma}\label{lem:canonical-new}
Suppose that $\mathscr{C}$ is a canonical set in a Brouwer algebra 
$\mathscr{B}$. If $\{a_i: i\in I\}$ and $\{b_j: j\in J\}$ are elements of 
$\mathscr{C}$ then we have: 
\begin{enumerate}
  \item $\bigwedge_{i \in I} a_i \leq \bigwedge_{j \in J} b_j 
      \Leftrightarrow \QA{j\in J}\QE{i \in I}[a_i \leq b_j]$; 
  \item $\bigwedge_{i\in I} a_i \rightarrow \bigwedge_{j \in J} b_j = 
      \bigvee_{i\in I}\left( \bigwedge_{j \in J} (a_i \rightarrow 
      b_j)\right)$. 
\end{enumerate}
\end{lemma}
\begin{proof} 
Item~(1) is obvious by meet-irreducibility of the elements of $\mathscr{C}$. 
As to item~(2), it is known (see for instance \cite[IX~\S2.~Theorem~3(ix)]{BalbesDwinger} in the dual context of Heyting 
algebras) that in any Brouwer algebra we have that $a\wedge b \rightarrow c= 
(a \rightarrow c) \lor (b \rightarrow c)$. Indeed, by taking infima and 
distributivity, from $a \lor (a\rightarrow c)\lor (b \rightarrow c)\ge c$ and 
$b \lor (a\rightarrow c)\lor (b \rightarrow c)\ge c$ one gets $(a\wedge 
b)\lor ((a\rightarrow c) \lor (b \rightarrow c))\ge c$, thus $a \wedge b 
\rightarrow c \leq (a\rightarrow c)\lor (b \rightarrow c)$; on the other 
hand, $a \lor (a\wedge b \rightarrow c) \ge (a \wedge b)\lor (a\wedge b 
\rightarrow c) \ge c$, from which $a \rightarrow c \leq a\wedge b \rightarrow 
c$, and similarly $b \rightarrow c \leq a\wedge b \rightarrow c$, which gives 
$(a \rightarrow c) \lor (b \rightarrow c) \leq a\wedge b \rightarrow c$. 

Putting things together, and using that $a \rightarrow b \wedge c= (a 
\rightarrow b) \wedge (a \rightarrow c)$ if $a\in \mathscr{C}$, we conclude 
\[
\bigwedge_{i\in I} a_i \rightarrow \bigwedge_{j \in J} b_j = 
      \bigvee_{i\in I} \left(a_i \rightarrow \bigwedge_{j \in J} b_i\right)= 
      \bigvee_{i\in I} \bigwedge_{j \in J} (a_i \rightarrow b_j). 
\] 
This shows item (2).
\end{proof}

\begin{lemma}\label{lem:free-can}
Let $\mathscr{U}$ be a finite implicative upper semilattice. Then the  
elements of $\mathscr{U}$ that freely generate 
$Fr^{L_d}_{\Usl_{\rightarrow}}(\mathscr{U})$ correspond exactly to the 
meet-irreducible elements of $Fr^{L_d}_{\Usl_{\rightarrow}}(\mathscr{U})$, 
and they form a canonical set in $Fr^{L_d}_{\Usl_{\rightarrow}}(\mathscr{U})$. 
\end{lemma}

\begin{proof}

We hope that a short description of $Fr^{L_d}_{\Usl_{\rightarrow}} 
(\mathscr{U})$ is in order. Let $\mathscr{U}=\langle U, \lor_U, \rightarrow_U \rangle$
be an implicative upper semilattice with partial ordering relation 
$\leq_U$. For a subset $X \subseteq U$, denote for simplicity 
$[X)=X_{u\leq_U}$, i.e.\ the upwards $\leq_U$-closure of $X$ in $\mathscr{U}$, 
which is a generic element of $\up(U)$. The lowercase Latin letters $x, y, z$ 
will vary through the elements of $U$, and we will write $[x)$ for $[\{x\})$. 

Let $Fr_U$ be the set of the formal expressions that one can obtain from the 
operation symbol $\wedge$, and the elements of $\mathscr{U}$ of the form 
$[x)$ for $x \in U$, i.e., 
\[
Fr_U=\{\bigwedge_{x \in X} [x): X \subseteq U,\; X \textrm{ finite}\},
\]
Using the $\leq_U$ of $\mathscr{U}$, define on $Fr_U$ the following
preordering relation for finite subsets $X,Y \subseteq U$:
\[
\bigwedge_{x \in X} [x) \leq \bigwedge_{y \in Y} [y) \Leftrightarrow
\QA{y \in Y}\QE{x \in X}\big[x \leq_U y].
\]
It is not difficult to see that the equivalence relation induced by this 
preordering partitions $Fr_U$ into equivalence classes that form a 
distributive lattice with $0$ and $1$, where the lattice operations of meet 
and join are defined unambiguously on representatives of the equivalence 
classes by   
\begin{align*}
\bigwedge_{x \in X} [x) \wedge \bigwedge_{y \in Y} [y)
&= \bigwedge_{z \in X \cup Y}[z),\\
\bigwedge_{x \in X} [x) \lor \bigwedge_{y \in Y} [y)
&= \bigwedge_{x \in X, y \in Y} [x \lor_U y).
\end{align*}
where $\lor_U$ is the join operation of $\mathscr{U}$.

We claim that this lattice (a sub-upper semilattice of $\up(\mathscr{U})$) is 
the desired free distributive lattice  $Fr^{L_d}_ 
{\Usl_{\rightarrow}}(\mathscr{U})$. To verify this, as well as the additional 
claims of the lemma, in order to simplify notation we will work directly with 
the elements of $Fr_U$, rather than with their equivalence classes. One can 
easily see that the function $\iota(x)=[x)$ is an upper semilattice embedding 
of $\mathscr{U}$ into $Fr^{L_d}_{\Usl_{\rightarrow}} (\mathscr{U})$. The 
elements of $Fr_U$ are of the form $\bigwedge_{x \in X} [x)$ with $X$ 
finite, so that the meet-irreducible elements of $Fr^{L_d}_ 
{\Usl_{\rightarrow}}(\mathscr{U})$ are exactly the elements in the image of 
the embedding $\iota$. With their meets they generate 
$Fr^{L_d}_{\Usl_{\rightarrow}} (\mathscr{U})$. Moreover, $Fr^{L_d}_ 
{\Usl_{\rightarrow}}(\mathscr{U})$ and $\iota$ satisfy the universal property 
of free objects: if $\mathcal{G}=\langle G, \wedge_G, \lor_G \rangle$ is a 
distributive lattice and $g: \mathscr{U} \rightarrow \mathcal{G}$ is an upper 
semilattice homomorphism, then there exists a unique homomorphism $g': 
Fr^{L_d}_ {\Usl_{\rightarrow}}(\mathscr{U}) \rightarrow \mathcal{G}$ such 
that $g' \circ \iota=g$, namely $g'(\bigwedge_{x \in X} [x))=\bigwedge_G 
\{g(x): x \in X\}$. To see for instance that $g'$ preserves least upper 
bounds, we observe 
\begin{align*}
g'\left(\bigwedge_{x \in X} [x) \lor \bigwedge_{y \in Y} [y)\right)
&=g'\left(\bigwedge_{x \in X, y \in Y} [x_i \lor_U y)\right)\\
& = \bigwedge_G \{g(x \lor_U y): x \in X, y \in Y\}\\
& = \bigwedge_G \{g(x) \lor_G g(y): x \in X, y \in Y\}\\
& = \bigwedge_G \{g(x): x \in X\} \lor_G \bigwedge_G \{g(y): y \in Y\}\\
& = g'\left(\bigwedge_{x \in X} [x)\right) \lor_G g'\left(\bigwedge_{y \in Y} [y)\right).
\end{align*}
Although the operation of implication plays no role in the definition of 
$g'$, note that if $Fr^{L_d}_ {\Usl_{\rightarrow}}(\mathscr{U})$ is finite, 
then it is a Brouwer algebra since it is a finite distributive lattice. A 
closer inspection shows that in this case $\iota$ preserves the operation of 
implication as well. Indeed, if $\bigwedge_{z \in Z} [z) \in Fr_U$ then as $[y)$ is meet irreducible and $\bigwedge_{z \in Z}[x) \lor  [z)=\bigwedge_{z \in Z}[x \lor_U z)$ we have
\begin{align*}
[y) \leq [x) \lor \bigwedge_{z \in Z} [z) 
      &\Leftrightarrow \QA{z \in Z} [y \leq_U x \lor_U z]\\
      &\Leftrightarrow \QA{z \in Z}[x\rightarrow_U y \leq_U z]\\
      &\Leftrightarrow [x \rightarrow_U  y) \leq \bigwedge_{z \in Z} [z) .
\end{align*}
This shows that $\iota(x \rightarrow_U  y)=\iota(x) \rightarrow \iota(y)$, 
where the implication on the right side denotes the arrow operation of 
$Fr^{L_d}_{\Usl_{\rightarrow}} (\mathscr{U})$. Therefore the generators of 
$Fr^{L_d}_ {\Usl_{\rightarrow}}(\mathscr{U})$ (the elements of the form 
$\iota(x)$, for $x \in U$) satisfy properties (2) and (3) of the definition 
of a canonical set. 
  
We must show at this point that (1) of the definition of a canonical set is 
satisfied. For the sake of further simplifying notation, in the rest of the 
proof we identify the elements of $\mathscr{U}$ with their images in 
$Fr^{L_d}_{\Usl_{\rightarrow}}(\mathscr{U})$ under the embedding $i$, and we 
use lowercase Latin letters as variables for the elements of 
$Fr^{L_d}_{\Usl_{\rightarrow}}(\mathscr{U})$. Property (1) amounts to showing 
that if $a\in \mathscr{U}$, then for every $b,c$ we have $a\rightarrow b\wedge 
c= (a\rightarrow b)\wedge (a\rightarrow c)$.  Suppose that $a \rightarrow b 
\wedge c=\bigwedge_{i \in I} d_i$ where each $d_i$ is meet-irreducible (it is 
well known that in a finite distributive lattice any element can be uniquely 
represented as the meet of an antichain of meet-irreducible elements). We 
have 
\[
b\wedge c \leq a \lor \bigwedge_{i \in I} d_i= \bigwedge_{i \in I} (a \lor d_i).
\] 
Now, each $a \lor d_i$ is meet-irreducible, being an element of $\mathscr{U}$, 
hence $b \leq a \lor d_i$ or $c \leq a \lor d_i$. Therefore, we can pick two 
subsets $J, K \subseteq I$ such that $b \leq \bigwedge_{j \in J} (a \lor 
d_j)=a \lor  \bigwedge_{j \in J} d_j$ and $c \leq \bigwedge_{k \in K} (a \lor 
d_k)=a \lor  \bigwedge_{k \in K} d_k$. Hence $a\rightarrow b \leq 
\bigwedge_{j \in J} d_j$ and $a\rightarrow c \leq \bigwedge_{k \in K} d_k$, 
giving $(a \rightarrow b) \wedge (a \rightarrow c)\leq \bigwedge_{i \in I} (a 
\lor d_i)=  a\rightarrow b\wedge c$. On the other hand, by distributivity and 
easy calculations we get $b \wedge c \leq a \lor ((a \rightarrow b) \wedge (a 
\rightarrow c))$, which gives $a \rightarrow b \wedge c \leq (a \rightarrow 
b) \wedge (a \rightarrow c)$, thus showing property (1) of the definition of 
a canonical set. 
\end{proof}  

Let now $I_n=\{1,2, \ldots,n\}$, and, for any given set $X$, let 
$\mathcal{P}(X)$ denote the power set of $X$, and let $\mathscr{P}(X)= 
\langle \mathcal{P}(X), \supseteq \rangle$. 

\begin{definition}
Let $\mathscr{B}_n$  denote the Brouwer algebra  
$Fr^{L_d}_{\Usl_{\rightarrow}} (\mathscr{P}(I_n))$. 
\end{definition}

The intermediate propositional logic given by $\bigcap_{n>0} 
\Th(\mathcal{B}_n)$ is known as the \emph{Medvedev logic}. The following 
holds: 

\begin{lemma}\label{lem:free}
$\bigcap_{n>0} \bigcap_{x \in \mathscr{B}_n}
      \Th([0,x]_{\mathscr{B}_n})= \IPC$.
\end{lemma}

\begin{proof}
See \cite[Lemma~3]{Skvortsova:intuitionism}, the remark following it, and 
\cite[Propositions~4.6~and~4.7]{Kuyper-Medvedev}. 
\end{proof}

\subsection{The computability-theoretic  ingredients}
\label{ssct:ingredients} We now isolate an important class of Dyment degrees, 
which constitute s canonical set in $\mathscr{M}_D$. 

\begin{definition}
A $D$-mass problem $\MP{A} \subseteq \MP{P}$ is said to be 
\emph{Dyment-Muchnik} if $\MP{A}$ is nonempty and $\MP{A}$ is upwards 
$\leq_e$-closed. A \emph{Dyment-Muchnik} degree is the degree of some 
Dyment-Muchnik mass problem. 
\end{definition}

\begin{lemma}\label{lem:can1}
The Dyment-Muchnik degrees form a canonical set in $\mathscr{M}_D$.
\end{lemma}

\begin{proof}
With three sublemmata we will verify the three conditions in the definition 
of a canonical set. The analogue of Sublemma~\ref{sublem:1} for the Medvedev 
lattice was independently proved by Skvortsova~\cite{Skvortsova:intuitionism} 
and Sorbi~\cite{Sorbi:Embedding}. Throughout the proofs of 
Sublemma~\ref{sublem:1} and Sublemma~\ref{sublem:2}, we will denote by 
$\mathcal{P}_F$ the set of \emph{finite} partial functions. 

\begin{sublemma}\label{sublem:1}
Let $\MP A, \MP{B}_0, \MP{B}_1, \MP Z$ be $D$-mass problems, with $\MP A$ 
being Dyment-Muchnik. If $\MP A \oplus \MP Z \geq_D \MP{B}_0 \meet \MP{B}_1$, 
then there are $\MP{Z}_0, \MP{Z}_1 \subseteq \MP Z $ such that 
\begin{itemize}
\item $\MP Z \equiv_D \MP{Z}_0 \meet \MP{Z}_1$ and
\item $\MP A \oplus \MP{Z}_i \geq_D \MP{B}_i$ for $i = 0,1$.
\end{itemize}
Hence, we have 
\[
\MP A \imp (\MP{B}_0 \meet \MP{B}_1) \equiv_D (\MP A \imp \MP{B}_0)
                     \meet (\MP A \imp \MP{B}_1).
\]
\end{sublemma}

\begin{proof}[Proof~of~Sublemma~\ref{sublem:1}] 
Let $\Psi$ be an enumeration operator such that $\MP A \oplus \MP Z \subseteq 
\dom(\Psi)$, and $\Psi(\MP A \oplus \MP Z) \subseteq 0^\smf \MP{B}_0 \cup 
1^\smf \MP{B}_1$.  

For $i = 0,1$, let
\begin{align*}
\MP{Z}_i = \{\chi \in \MP{Z} : (\exists \alpha \in \MP{P}_F)[\langle 0, i 
\rangle \in \Psi(\alpha \oplus \chi)]\}.
\end{align*}

Notice that $\MP{Z} \leq_D \MP{Z}_0 \andd \MP{Z}_1$ since $\MP{Z}_0, \MP{Z}_1 
\subseteq \MP Z$. 

Fix $\chi \in \MP{Z}$, and let $\psi \in \MP{A}$. Then $\Psi(\psi \oplus 
\chi)\downarrow$ and $\Psi(\psi \oplus \chi)(0) \downarrow \in \{0,1\}$. 
Suppose that $\Psi(\psi \oplus \chi)(0) \downarrow =0$. We will show in this 
case that for every $\alpha\in \MP{P}_F$, and every $y$, if $\langle 0, y 
\rangle \in \Psi(\alpha \oplus \chi)$ then $y=0$. Given $\alpha, \beta \in 
\MP{P}$, let $\psi^\alpha=\psi\smallsetminus \{\langle x,y \rangle: x \in 
\dom(\alpha)\}$ and $\psi^{\alpha,\beta}=(\psi^\alpha)^{\beta}$. By our 
assumptions there exists $\alpha_0\in \MP{P}_F$ such that $\langle 0, 0 
\rangle \in \Psi(\alpha_0 \oplus \chi)$. Assume, by contradiction, that there 
exist $\alpha_y\in \MP{P}_F$ and $y\ne 0$ such that $\langle 0,y\rangle \in 
\Psi(\alpha_y \oplus \chi)$. Now, $\psi^{\alpha_0, \alpha_y}\equiv_e \psi$, 
and since $\MP{A}$ is Dyment-Muchnik, we have that $\psi^{\alpha_0, 
\alpha_y}\in \MP{A}$, thus $\Psi(\psi^{\alpha_0, \alpha_y}\oplus 
\chi)\downarrow$, and thus for some $z\in \{0,1\}$ there exists $\beta_z \in 
\MP{P}_F$ such that $\beta_z \subseteq \psi^{\alpha_0, \alpha_y}$ and 
$\langle 0, z \rangle \in \Psi(\beta_z \oplus \chi)$. But this yields a 
contradiction: if $z=0$ then $\psi^{\alpha_0, \alpha_y} \cup \alpha_y$  is a 
partial function (since $\dom(\psi^{\alpha_0, \alpha_y}) \cap 
\dom(\alpha_y)=\emptyset$) which lies in $\MP{A}$ (since $\psi^{\alpha_0, 
\alpha_y} \cup \alpha_y\equiv_e \psi^{\alpha_0, \alpha_y}$), but 
$\Psi((\psi^{\alpha_0, \alpha_y} \cup \alpha_y)\oplus \chi)$ is not 
single-valued (since both $\langle 0, 0\rangle$ and $\langle 0, y\rangle$ lie 
in $\Psi((\psi^{\alpha_0, \alpha_y} \cup \alpha_y)\oplus \chi)$, being 
$\beta_z, \alpha_y \subseteq \psi^{\alpha_0, \alpha_y} \cup \alpha_y$), 
contradicting $(\psi^{\alpha_0, \alpha_y} \cup \alpha_y)\oplus \chi \in 
\dom(\Psi)$. If $z=1$ then a similar argument leads to showing that 
$\psi^{\alpha_0, \alpha_y} \cup \alpha_0$ lies in $\MP{A}$ but 
$(\psi^{\alpha_0, \alpha_y} \cup \alpha_0)\oplus \chi \notin \dom(\Psi)$. 

By symmetry, the previous argument shows that if we initially choose $\psi 
\in \MP{A}$ such that $\Psi(\psi \oplus \chi)(0) \downarrow =1$, then for 
every $\alpha\in \MP{P}_F$, and every $y$, if $\langle 0, y \rangle \in 
\Psi(\alpha \oplus \chi)$ then $y=1$. 

Let $\Gamma$ be the following enumeration operator: for $\chi \in 
\mathcal{P}$,  
\[
\Gamma(\chi) =\{\langle i, y\rangle: i>0 \,\&\,  
                \langle i-1,y\rangle \in \chi\} \cup 
\{\langle 0, y\rangle: \QE{\alpha \in \MP{P}_F}[\langle 0, y \rangle \in 
                \Psi(\alpha\oplus \chi)]\}.
\]
By the previous remarks it is clear that $\MP{Z} \subseteq \dom(\Gamma)$, and 
for every $\chi \in \mathcal{Z}$, we have that $\Gamma(\chi)(0)\downarrow \in 
\{0,1\}$. Therefore $\MP{Z} \geq_D \MP{Z}_0 \wedge \MP{Z}_1$ via $\Gamma$. 
The $e$-operator $\Phi(\psi \oplus \chi)$ such that $\Phi(\psi \oplus 
\chi)(y)=\Psi(\psi \oplus \chi)(y+1)$ for each $y$ provides a reduction 
$\MP A \oplus \MP{Z}_i \geq_D \MP{B}_i$ for both $i \in \{0,1\}$. 

To finish off, we must show that $\MP A \imp \MP{B}_0 \meet \MP{B}_1 \equiv_D 
(\MP A \imp \MP{B}_0) \meet (\MP A \imp \MP{B}_1)$. Well, in every Brouwer 
algebra we have $a\rightarrow b_0 \wedge b_1 \leq (a \rightarrow b_0) \wedge 
(a \rightarrow b_1)$ since $a \lor ((a \rightarrow b_0) \wedge (a \rightarrow 
b_1)) \geq b_0 \wedge b_1$. Therefore, $\MP{A} \rightarrow \MP{B}_0 \wedge 
\MP{B}_1 \leq_D (\MP{A} \rightarrow \MP{B}_0) \wedge (\MP{A} \rightarrow 
\MP{B}_1)$. On the other hand, by what proved in the first part of the 
sublemma taking $\mathcal{Z}$ to be $\MP A \imp \MP{B}_0 \meet \MP{B}_1$, 
from 
\[
\mathcal{B}_0 \wedge \mathcal{B}_1\leq_D \mathcal{A} \oplus 
(\mathcal{A} \rightarrow \mathcal{B}_0 \wedge \mathcal{B}_1)
\]
one gets that there exist $\mathcal{Z}_0, \mathcal{Z}_1$ such that $\MP A 
\imp \MP{B}_0 \meet \MP{B}_1 \equiv_D \MP{Z}_0 \meet \MP{Z}_1$, and $\MP{A} 
\rightarrow \MP{B}_i \leq_D \MP{Z}_i$ for $i= 0,1$. Hence 
\[
\MP A \imp \MP{B}_0 \meet \MP{B}_1 \geq_D (\MP A \imp \MP{B}_0) \wedge 
(\MP A \imp \MP{B}_1). \qedhere
\]
\end{proof}

\begin{sublemma}\label{sublem:2}
If $\MP{A}$ is a Dyment-Muchnik mass problem then its Dyment degree is 
meet-irreducible. 
\end{sublemma}

\begin{proof}[Proof~of~Sublemma~\ref{sublem:2}] 
Suppose that $\MP{A}$ is a Dyment-Muchnik mass problem and $\MP{A} \equiv_D 
\MP{B}_0 \wedge \MP{B}_1$. Let $\MP{A} \geq_D \MP{B}_0 \wedge \MP{B}_1$ via 
the enumeration operator $\Gamma$. We want to show that either 
$\MP{B}_0\leq_D \MP{A}$ or $\MP{B}_1\leq_D \MP{A}$. Let $\phi \in \MP{A}$. We 
claim that if $\Gamma(\phi)(0)\downarrow =0$ then $\Gamma(\MP{A}) \subseteq 
0^\smf \MP{B}_0$, and thus $\MP{B}_0 \leq_D \MP{A}$; and if 
$\Gamma(\phi)(0)\downarrow =1$  then $\Gamma(\MP{A}) \subseteq 1^\smf 
\MP{B}_1$, and thus $\MP{B}_1 \leq_D \MP{A}$.  

So, let us assume that $\Gamma(\phi)(0)\downarrow =0$: the case 
$\Gamma(\phi)(0)\downarrow =1$ is symmetric. So there exists $\alpha_0 \in 
\MP{P}_F$ with $\alpha_0 \subseteq \phi$ and $\langle 0, 0 \rangle \in 
\Gamma(\alpha_0)$. Suppose now that $\psi \in \mathcal{A}$ is another partial 
function such that $\Gamma(\psi)(0)\downarrow =1$, and let $\alpha_1 \in 
\MP{P}_F$ with $\alpha_1 \subseteq \psi$ and $\langle 0, 1 \rangle \in 
\Gamma(\alpha_1)$. Arguing as in the previous sublemma, we have that 
$\phi^{\alpha_0, \alpha_1} \in \mathcal{A}$, and thus $\Gamma(\phi^{\alpha_0, 
\alpha_1})(0)\downarrow \in \{0,1\}$. We may suppose that 
$\Gamma(\phi^{\alpha_0, \alpha_1})(0)\downarrow =0$, the other case being 
symmetric. But then $\{\langle 0,0\rangle, \langle 0,1\rangle\} \subseteq 
\Gamma(\phi^{\alpha_0, \alpha_1}\cup \alpha_1)$, which is a contradiction 
since $\phi^{\alpha_0, \alpha_1}\cup \alpha_1 \in \mathcal{A}$, and thus 
$\phi^{\alpha_0, \alpha_1}\cup \alpha_1 \in \dom(\Gamma)$.

It follows that $\Gamma(\psi) \in 0^\smf \MP{B}_0 \equiv_D \MP{B}_0$ for 
every $\psi \in \mathcal{A}$, as desired.    
\end{proof}

\begin{sublemma}\label{sublem:3}
On Dyment-Muchnik mass problems, the operation of least upper bound is given 
by intersection, which still yields a Dyment-Muchnik mass problem. If 
$\MP{B}$ is Dyment-Muchnik, then 
\[
\MP{A} \rightarrow \MP{B}= \{\psi:  \QA{\phi \in \mathcal{A}}[\phi \oplus 
\psi \in \mathcal{B}] \},
\]
which still is a Dyment-Muchnik mass problem. Hence the Dyment-Muchnik 
degrees form an sub-implicative upper semilattice of $\mathscr{M}_D$. 
\end{sublemma}

\begin{proof}[Proof~of~Sublemma~\ref{sublem:3}] 
To show the claim about $\rightarrow$, let $\mathcal{D}=\{\psi:  \QA{\phi \in 
\mathcal{A}}[\phi \oplus \psi \in \mathcal{B}]\}$. Clearly, 
$\mathcal{B}\leq_D \mathcal{A} \oplus \mathcal{D}$ since $\mathcal{A} \oplus 
\mathcal{D} \subseteq \mathcal{B}$. To show minimality of $\mathcal{D}$, 
assume that $\Phi$ is an enumeration operator reducing $\MP{B}$ to $\MP{A} \oplus 
\mathcal{Z}$. If $\psi \in \MP{Z}$ then for every $\phi \in 
\MP{A}$ we have $\Phi(\phi \oplus \psi)$ is defined and $\Phi(\phi \oplus 
\psi)\in \MP{B}$. But since $\MP{B}$ is Dyment-Muchnik, this implies that 
$\phi \oplus \psi\in \MP{B}$, and thus $\psi \in \MP{D}$, giving $\MP{Z} 
\subseteq \MP{D}$, and therefore $\MP{D} \leq_D \MP{Z}$. The claim about 
$\oplus$ being intersection is trivial. From this it follows that $\MP{A} 
\rightarrow \MP{B}$ and $\MP{A} \oplus \MP{B}$ are Dyment-Muchnik mass 
problems, if so are $\MP{A}$ and $\MP{B}$. 
\end{proof}

Putting the three sublemmata together we conclude that the Dyment-Muchnik 
degrees form a canonical set in $\mathscr{M}_D$. Indeed, 
Sublemma~\ref{sublem:1} shows that they satisfy item (1) of 
Definition~\ref{def:canonical}; Sublemma~\ref{sublem:2} shows that they 
satisfy item (2) of Definition~\ref{def:canonical}, and 
Sublemma~\ref{sublem:3} shows that they satisfy item (3) of 
Definition~\ref{def:canonical}. 
\end{proof}

\begin{corollary}\label{cor:free}
If $\{\mathbf{A}_i: i \in I\}$ and $\{\mathbf{B}_j: j \in J\}$ are 
Dyment-Muchnik degrees then, in $\mathscr{M}_D$, 
\[
\bigwedge_{i\in I} \mathbf{A}_i \rightarrow \bigwedge_{j \in J} \mathbf{B}_j =
      \bigvee_{i\in I}\bigwedge_{j \in J} (\mathbf{A}_i \rightarrow \mathbf{B}_j).
\] 
\end{corollary}
\begin{proof}
By Lemma~\ref{lem:canonical-new} and Lemma~\ref{lem:can1}.
\end{proof}

\begin{corollary}
Let $A$, $\{x_i: i \in I\}$, and $\alpha$ be as in Definition~\ref{def:alpha} 
relative to the pre-upper semilattice $\mathscr{U}=\mathscr{P}_e$. Then 
each $D$-mass problem $\alpha(X)$ is Dyment-Muchnik. 
\end{corollary}

\begin{proof}
By Lemma~\ref{lem:alphab}, the $D$-mass problem $\alpha(X)$ is 
upwards~$\leq_e$-closed. 
\end{proof}

\begin{corollary}\label{cor:all}
For every $n>0$, if $\{\psi_i: i \in I_n\}$ is a strong $u$--antichain in 
some downwards $\leq_e$-closed subset $\mathcal{A}$ of $\mathcal{P}$ and 
$\alpha: \mathcal{P}(I_n) \longrightarrow \up(\mathscr{P}_e)$ is as in 
Lemma~\ref{lem:algebraic}, then there exists an embedding  $\gamma: 
\mathscr{B}_n \xhookrightarrow{\mathbb{B}} \Big[[\alpha(I_n)]_D, 
[\alpha(\emptyset)]_D \Big]_{\mathscr{M}_D}$. 
\end{corollary}   

\begin{proof}
Let $n>0$ be given. Let $\iota: \mathscr{P}(I_n) \xhookrightarrow 
{\Usl_{\rightarrow}}\mathscr{B}_n$ be the $\Usl_{\rightarrow}$-embedding 
described in the proof of Lemma~\ref{lem:free-can}, i.e. $\iota(x)=[x)$ for 
$x \in \mathscr{P}(I_n)$. In analogy with the proof of 
Lemma~\ref{lem:free-can} we use the letters $x,y$ as variables for the 
elements of $\mathscr{P}(I_n)$, and $X,Y$ as variables for subsets of 
$\mathscr{P}(I_n)$. Then,  consider the embedding $\alpha: \mathscr{P}(I_n) 
\xhookrightarrow{\Usl_{\rightarrow}} \Big[\alpha(I_n), \alpha(\emptyset)] 
\Big]_{\up(\mathscr{P}_e)}$, as in Lemma~\ref{lem:algebraic}. Finally, let 
$d: \up(\mathscr{P}_e) \rightarrow \mathscr{M}_D$ be 
$d(\mathcal{A})=[\mathcal{A}]_D$. By the fact that the Dyment-Muchnik degrees 
form a canonical set in $\mathscr{M}_D$, the composition $d\circ \alpha$ is 
an $\Usl_{\rightarrow}$-embedding $d\circ \alpha: \mathscr{P}(I_n) 
\xhookrightarrow{\Usl_{\rightarrow}} \Big[[\alpha(I_n)]_D, 
[\alpha(\emptyset)]_D \Big]_{\mathscr{M}_D}$. Note that both $\iota$ and 
$d\circ \alpha$ preserve $0, 1$. By the freeness properties of 
$\mathscr{B}_n$ there is a unique morphism $\gamma: \mathscr{B}_n 
\overset{\mathbb{L}_d}{\longrightarrow} \Big[[\alpha(I)]_D, 
[\alpha(\emptyset)]_D \Big]_{\mathscr{M}_D}$, preserving $0,1$, and  such 
that $\gamma([x))=d(\alpha(x))=[\alpha(x)]_D$, for every  $x \in 
\mathscr{P}(I_n)$. 

It is only left to show that $\gamma$ preserves $\rightarrow$ and is an 
order-theoretic embedding. In the rest of proof, for notational simplicity we 
use the same symbols for the operations of $\mathscr{P}(I_n)$ and the 
corresponding operations of $\mathscr{M}_D$, leaving it to the reader to 
easily distinguish which denotes which. By Lemma~\ref{lem:can1}, every 
element of $\mathscr{B}_n$ has the form $\bigwedge_{x \in X} [x)$ for some $X 
\subseteq \mathscr{P}(I_n)$. Moreover, on generators $[x), [y)$ we have 
$\gamma([x)\imp [y))=\gamma([x)) \imp \gamma([y))$ since $\gamma([x) 
\rightarrow [y)))=d(\alpha(x \rightarrow y))$, but $x \rightarrow y$ is an 
element of $\mathscr{P}(I_n)$ and thus $\alpha(x \rightarrow y)=\alpha(x) 
\rightarrow \alpha(y)$, therefore $d(\alpha(x \rightarrow y))=d(\alpha(x) 
\rightarrow \alpha(y))=d(\alpha(x))\rightarrow d(\alpha(y))$ since 
$\alpha(x)$ and $\alpha(y)$ are Dyment-Muchnik mass problems. Therefore, 
using that $\gamma$ preserves $\wedge$ and $\lor$, we have, for subsets $X, Y 
\subseteq \mathscr{P}(I_n)$: 
\begin{align*}
\gamma\left(\bigwedge_{x \in X} [x) \rightarrow \bigwedge_{y \in Y} [y)\right) 
    &=\gamma\left(\bigvee_{x \in X}\bigwedge_{y \in Y} ([x) \rightarrow [y))\right) 
                           \qquad  \textrm{ (by Lemma~\ref{lem:canonical-new})},\\
    &=\bigvee_{x \in X} \bigwedge_{y \in Y} \gamma([x) \rightarrow [y))\\
    &=\bigvee_{x \in X} \bigwedge_{y \in Y} (\gamma([x)) \rightarrow \gamma([y)))\\
    &= \bigwedge_{x \in X}  \gamma([x)) \rightarrow 
           \bigwedge_{y \in Y} \gamma([y)) \qquad 
                     \textrm{(by Corollary~\ref{cor:free})}\\
    &= \gamma\left(\bigwedge_{x \in X} [x)\right) \rightarrow 
    \gamma\left(\bigwedge_{j\in J} [y)\right).
\end{align*} 
It follows that $\gamma$ preserves the arrow-operation, since every element 
in $\mathscr{B}_n$ is a meet of meet-irreducible elements. 

Finally, $\gamma$ is an order isomorphism. Indeed, if $X=\bigwedge_{x \in 
X}[x)$ and $Y=\bigwedge_{y \in Y}[y)$ are elements of $\mathscr{B}_n$ with 
$X,Y \subseteq \mathscr{P}(I_n)$, then by the meet-irreducibility of each $[x), 
[y)$ in $\mathscr{B}_n$, the meet-irreducibility in $\mathscr{M}_D$ of every 
$\gamma([x))$ (see property (2) of the Dyment-Muchnik degrees), and since 
$\gamma$ is an order-theoretic embedding on generators, we have 
\begin{align*}
\bigwedge_{x \in X}[x) \leq \bigwedge_{y \in Y} [y) &\Leftrightarrow
\QA{y\in Y}\QE{x \in X} \left[ [x) \leq [y) \right]\\
&\Leftrightarrow
\QA{y\in Y}\QE{x \in X} \left[\gamma([x)) \leq_D \gamma([y))\right]\\
&\Leftrightarrow \bigwedge_{x \in X} \gamma([x)) \leq_D \bigwedge_{y \in Y} 
\gamma([y))\\
&\Leftrightarrow \gamma(\bigwedge_{x \in X} [x))
             \leq_D \gamma(\bigwedge_{y \in Y}[y)). \qedhere
\end{align*} 
\end{proof} 

\subsection{There is a Dyment degree $\mathbf{Z}$ such that 
$\Th([\mathbf{0}_D, \mathbf{Z}]_{\Dym})=\IPC$} 

Our $\mathbf{Z}=[\MP{Z}]_D$ will be chosen so that it is possible to implement the
following plan.  Fix $n > 0$.  In view of Corollary~\ref{cor:all}, we will look for a suitable strong 
$u$--antichain in some downwards~$\leq_e$-closed set of 
partial functions so that, by Corollary~\ref{cor:all}, there is an embedding 
$\gamma: \mathscr{B}_n  \xhookrightarrow{\mathbb{B}} [[\alpha(I_n)]_D, 
[\alpha(\emptyset)]_D]_{\Dym}$. Let us denote $\alpha(I_n)=\mathcal{X}_n$ 
and $\alpha(\emptyset)=\mathcal{Y}_n$. The strong $u$--antichain and therefore the
embedding will be arranged so that for $X \in \mathscr{B}_n$ 
(say $X=\bigwedge_{k \in K} [x_k)$, where each $x_k \in \mathscr{P}(I_n)$), 
in view of Fact~\ref{fact:1}, we will have
\[
\mathcal{Y}_{n,X}{\equiv_D}  \mathcal{Z} \oplus  \mathcal{X}_{n},
\]
where $\mathcal{Y}_{n,X}= \bigwedge_{k \in K} \alpha(x_k)$ and so
$\gamma(X)=[\mathcal{Y}_{n,X}]_D$. Letting $\mathbf{X}_n=[\mathcal{X}_n]_D$ 
and $\mathbf{Y}_{n,X}=[\mathcal{Y}_{n,X}]_D$, we 
have $\mathbf{Y}_{n,X}= \mathbf{Z} \lor \mathbf{X}_n$, and from 
Fact~\ref{fact:1} it follows that $\Th([\mathbf{X}_{n}, 
\mathbf{Y}_{n,X}]_{\Dym}) \subseteq \Th([0,X]_{\mathscr{B}_n})$, and thus by 
Lemma~\ref{lem:free}
\begin{align*}
  \Th([\mathbf{0}_D, \mathbf{Z}]_{\Dym}) & \subseteq \bigcap_{n>0} \bigcap_{X \in \mathscr{B}_n}
\Th\big([\mathbf{X}_{n}, \mathbf{Y}_{n,X}]_{\Dym}\big) \\
   & \subseteq
\bigcap_{n>0} \bigcap_{X \in \mathscr{B}_n}
      \Th\big([0,X]_{\mathscr{B}_n}\big)\\
      &= \IPC.
\end{align*}

In order to find a $D$-mass problem $\mathcal{Z}$ as required above, we 
still need a couple of computability-theoretic results due to Kuyper, which 
we review in the next two lemmas. In the rest of this section, if $g$ is a 
total function, then by $g'$ we mean a total function which is $\equiv_T$ to 
the Turing jump of $g$. 

\begin{lemma}\label{lem:main21}\cite[Theorem~4.10]{Kuyper-Medvedev}
Suppose that $p,q$ are total functions such that $p' \le_e q$, and let $k_0, 
k_1, \ldots$ be total functions which are uniformly $e$-reducible to $q$ and 
$k_i \nleq_e p$. Then there exists a total function $f$ such that $p\le_e f$, 
$f' \le_e q$ and for every $i$, $q \le_e f \oplus k_i$. 
\end{lemma}

\begin{proof}
See \cite{Kuyper-Medvedev} where the statement is proved for sets and Turing 
reducibility. Our claim then follows from the fact that on total functions 
$\le_T$ can be uniformly copied to $\leq_e$.
\end{proof}

If $\phi \in \MP{P}$ then by $\phi^{[i]}$ we mean the partial function 
$\phi^{[i]}(x)=\phi(\langle i,x\rangle)$; notice that if $f$ is a total 
function then $f^{[i]}$ is total for every $i$. Furthermore, a partial 
function $\phi$ is called \emph{computably $e$-independent} if for every $i$, 
$\phi^{[i]} \nleq_e \bigoplus_{j\ne i} \phi^{[j]}$, where we understand that 
if $I\subseteq \omega$ is a set of indices then $\bigoplus_{i \in I} 
\phi^{[i]}$ is the partial function such that $\bigoplus_{i \in I} 
\phi^{[i]}(\langle j,x\rangle)=\phi_j(x)$ if $j \in I$, and $\bigoplus_{i \in 
I} \phi^{[i]}(\langle j,x\rangle)=0$ if $j \notin I$. 

\begin{lemma}\cite[Lemma~4.11]{Kuyper-Medvedev}\label{lem:main22}
Let $g$ be a computably $e$-independent total function. Then there exists a 
total function $h$ such that $g \oplus h$ is computably $e$-independent, 
where we understand that the $2i$-th column of $g \oplus h$ is $g^{[i]}$ and 
the $2i+1$-th column is $h^{[i]}$. 
\end{lemma}

\begin{proof}
See \cite{Kuyper-Medvedev} where a full proof is given for sets and Turing 
reducibility. 
\end{proof}

\begin{theorem}\label{thm:main2}
If $g$ is a computably $e$-independent total function, then $\Th(\left[ 
\mathbf{0}_{D}, \mathbf{Z}\right]_{\Dym}) =\IPC$, where 
$\mathbf{Z}=[\MP{Z}]_D$ and $\MP{Z}=\{\cat{i}{\phi}: i \in \omega, 
g^{[i]}\le_{e} \phi\}$.
\end{theorem}

\begin{proof}
The proof is as in \cite[Theorem~1.1]{Kuyper-Medvedev}, but replacing $\le_T$ 
with $\le_e$. For the sake of completeness and the ease of the reader we give 
the details as follows. Suppose that $g$ is total and computably 
$e$-independent. Given $n >0 $ and $x_1, \ldots, x_k \in \mathscr{P}(I_n)$, 
let us show first that there exist a downwards~$\le_e$-closed $D$-mass 
problem $\mathcal{A}$ and a strong $u$--antichain of total functions $f_1, 
\ldots, f_n \in \mathcal{A}$ such that (in the notation of the proof of 
Corollary~\ref{cor:all} but taking $I=I_n$), 
\begin{equation}\label{eq:main}
\left(\aaa^c \cup \{f_1, \ldots, f_n\}_{u\le_e} \right)
                 \oplus  \left\{\cat{i}{\phi}:
i \in \omega, g^{[i]}\le_e \phi\right\}\equiv_D \bigwedge_{1 \le j \le k}
\left(\aaa^c \cup \{f_i: i\in x_j\}_{u\le_e} \right),
\end{equation}
i.e. 
\[
\alpha(I_n) \oplus \mathcal{Z}\equiv_D \bigwedge_{1 \le j \le k}\alpha(x_j). 
\]

By fixing $n>0$ and $x_1, \ldots, x_k \in \mathscr{P}(I_n)$ we have just 
fixed a pair $n, X$, with $n>0$ and $X \in \mathscr{B}_n$, where 
$X=\bigwedge_{1 \le j \le k} [x_j)$, so that 
\begin{align*}
\mathcal{Z}&=\left\{\cat{i}{\phi}: i \in \omega, g^{[i]}\le_e \phi\right\}\\
\mathcal{X}_{n} &=\aaa^c \cup \{f_1, \ldots, f_n\}_{u\le_e}\\
\mathcal{Y}_{n,X} &= \bigwedge_{1 \le j \le k}\alpha(x_j),
\end{align*}
obtaining $\mathcal{X}_{n} \oplus \mathcal{Z}\equiv_D \mathcal{Y}_{n,X}$. 

Since $[0, X]_{\mathscr{B}_n} \xhookrightarrow{\mathbb{B}} [\mathbf{X}_{n}, 
\mathbf{Y}_{n,X}]_{\mathscr{M}_D}$, we have $\Th\big([\mathbf{X}_{n}, 
\mathbf{Y}_{n,X}]_{\mathscr{M}_D}\big) \subseteq \Th([0, 
X])_{\mathscr{B}_n})$. 

To show the existence of $\MP{A}$, by Lemma~\ref{lem:main22} let $h$ be a
total function such that $g \oplus h$ is computably $e$-independent. Define
\[
\aaa=\big(\{(g \oplus h)'\}_{u\le_e}\big)^c
\]
(where complementation is taken with respect to the set of all partial 
functions) so that $\aaa$ is clearly downwards $\le_e$-closed. Given $i$ with 
$1 \le i \le n$, apply Lemma~\ref{lem:main21} with $p = \Big(\bigoplus_{1\le 
j \le k, i \in x_j} g^{[j]}\Big) \oplus h^{[i]}$, $q = (g\oplus h)'$, and the 
sequence of $k$'s listing the $g^{[j]}$ with either $i \notin x_j$ or $j > k$ 
together with the $h^{[j]}$ with $j \neq i$.  The result is a total function 
$f_i$ such that the following hold: 
\begin{align*}
&\Big( \bigoplus_{1\le j \le k, i \in x_j} g^{[j]}\Big) \oplus h^{[i]} \le_e 
    f_i,\\
& f_i' \le_e (g\oplus h)',\\
&(\forall j)\Big[ \big[ [1 \le j \le k \,\&\, i \notin 
    x_j] \lor j >k \big] \Rightarrow (g \oplus h)' \le_e f_i 
\oplus g^{[j]}\Big],\\ 
&(\forall j \ne i)\big[(g \oplus h)' \le_e f_i \oplus h^{[j]}\big]. 
\end{align*}

We now show that $\{f_1, \ldots, f_n\}$ is a strong $u$--antichain in $\aaa$. 
It is immediate to see that $f_i \in \aaa$ for every $1 \le i \le n$. If 
$1\le i <j \le n$ then 
\[
(g\oplus h)' \le_e h^{[i]} \oplus f_j \le_e f_i \oplus f_j,
\]
so that $f_i \oplus f_j \notin \aaa$.

Finally, we show equation~(\ref{eq:main}). To see $\ge_D$, assume that 
\[
\psi \oplus \chi \in \left(\aaa^c \cup \{f_1, \ldots, f_n\}_{u\le_e} 
\right) \oplus  \left\{ \cat{i}{\phi}: g^{[i]}\le_e \phi\right\}:
\]
in particular let $j$ be such that $\chi=\cat{j}{\phi}$ and $g^{[j]} \le_e 
\phi$. We split the verification into cases: 
\begin{enumerate}
  \item Assume first that $j=0$ or $j>k$. We show in this case that $\psi 
      \oplus \phi \in \aaa^c \subseteq \alpha(x_1)$. Indeed, if $\psi \ge_e 
      f_i$ for some $1\le i \le n$ then 
\[
(g\oplus h)' \le_e f_i \oplus g^{[j]} \le_e \psi \oplus \phi.
\]
On the other hand, if $\psi \in \aaa^c$, i.e. $(g\oplus h)'\le_e \psi$, 
        then $(g\oplus h)' \le_e \psi \oplus \phi$.
        
In either case we have that $\psi \oplus \phi \in \aaa^c$ and therefore the 
mapping $\psi \oplus (\cat{j}{\phi}) \mapsto \cat{1}{(\psi \oplus \phi)}$ 
maps $\psi \oplus (\cat{j}{\phi})$ to an element in 
$\cat{1}{(\alpha(x_1))}$. 

\item Assume now that $1 \leq j \le k$. In this case we show that $\psi 
    \oplus \phi \in \alpha(x_j)$. Indeed, if $\psi \ge_e f_i$ for some $i 
    \in x_j$ then $f_i \le_e \psi \oplus \phi$, but then we can use the 
    fact that $\{f_i\}_{u\le_e} \subseteq \alpha(x_j)$. On the other hand, 
    if $\psi \ge_e f_i$ for some $i \notin x_j$ then 
    \[ 
    (g\oplus h)' \le_e  f_i  \oplus g^{[j]} \le_e 
    \psi \oplus \phi\in \alpha(x_j).
    \] 
In either case, we have that the mapping  $\psi \oplus (\cat{j}{\phi}) 
\mapsto \cat{j}{(\psi \oplus \phi)}$ maps $\psi \oplus (\cat{j}{\phi})$ to 
an element in $\cat{j}{(\alpha(x_j))}$. 
    
Finally, if $(g\oplus h)'\le_e \psi$ then trivially  $(g\oplus h)'\le_e 
    \psi\oplus \phi$, and thus again $\psi\oplus \phi \in \alpha(x_j)$. 
    Even in this case, we have that the mapping  $\psi \oplus 
    (\cat{j}{\phi}) \mapsto \cat{j}{(\psi \oplus \phi)}$ maps $\psi \oplus 
(\cat{j}{\phi})$ to an element of $\cat{j}{(\alpha(x_j))}$. 
\end{enumerate}
So, uniformly in $j$ such that $\chi(0)=j$ (if $\chi(0)\downarrow$ and this 
is certainly the case if $\chi \in \{ \cat{i}{\phi}: g^{[i]}\le_e \phi\}$) we 
can map, via an enumeration operator, the partial function $\psi \oplus \chi$ 
to some partial function in $\cat{r}{(\alpha(x_r))}$, for some $r\in \{1, 
\ldots, k\}$. 

We still have the prove $\le_D$, i.e
\[ 
\left(\aaa^c \cup \{f_1, \ldots, f_n\}_{u\le_e} \right) \oplus 
\left\{\cat{i}{\phi}: i \in \omega, g^{[i]}\le_e \phi\right\} \le_D
 \bigwedge_{1 \le j \le k}\alpha(x_j).
\]
But this is obvious as $\{f_r: r \in x_j\}_{u\le_e} \subseteq \{f_1, \ldots, 
f_n\}_{u\le_e}$ and $\{f_r: r \in x_j\}_{u\le_e} \subseteq \left\{\phi: 
g^{[j]}\le_e \phi \right\}$, being $g^{[j]}\le_e f_r$ for every $r \in x_j$. 
So the mapping $\cat{j}{\phi} \mapsto \phi \oplus (\cat{j}{\phi})$ induces an 
enumeration operator yielding the desired reduction. 
\end{proof}

\section{The Dyment-Muchnik lattice and intermediate 
logics: An initial segment of the Dyment-Muchnik lattice modeling exactly 
$\IPC$}\label{sec:Dyment-Muchnik} 

We begin by giving a convenient characterization of $\mathscr{M}_{D,w}$. A 
\emph{mass problem of sets} is a subset $\MP{A}\subseteq 2^\omega$. On mass 
problems $\mathcal{A}$ and $\mathcal{B}$ of sets consider the relations 
$\leq_{D}$ and $\leq_{D,w}$ (for which we use the same symbols as for 
Dyment reducibility and Dyment-Muchnik reducibility on mass problems of 
partial functions) given respectively by 
\[
\mathcal{A} \leq_{D} \mathcal{B} \Leftrightarrow \QE{ \textrm{$e$-operator $\Phi$}}
[\Phi(\mathcal{B}) \subseteq \mathcal{A}],
\]
and
\[
\mathcal{A} \leq_{D,w} \mathcal{B} \Leftrightarrow
\QA{B\in \mathcal{B}}\QE{A\in 
\mathcal{A}}[A \leq_e B].
\]
These relations induce preordering relations, and thus, degree structures 
$\mathscr{M}_{D}^{set}$ and $\mathscr{M}_{D,w}^{set}$, on the collection of 
mass problems of sets. Now, the functions induced on $e$-degrees by the two 
mappings $\graph: \mathcal{P} \longrightarrow 2^\omega$ and $s: 2^\omega 
\longrightarrow  \mathcal{P}$, given by, respectively, $\phi \mapsto 
\graph(\phi)$ and $A \mapsto A\times\{0\}$, are easily seen to be 
order-theoretic homomorphisms which invert each other up to degree and 
therefore they yield an isomorphism $F: \mathscr{D}^{pf}_e \simeq 
\mathscr{D}_e^{set}$, showing that $\mathscr{M}_{D} \simeq 
\mathscr{M}_{D}^{set}$ and $\mathscr{M}_{D,w} \simeq \mathscr{M}_{D,w}^{set}$ 
via the isomorphisms induced in both cases by the assignment $\mathcal{A} 
\mapsto \{ \graph(\phi): \phi \in \mathcal{A}\}$, for every mass problem 
$\mathcal{A} \subseteq \mathcal{P}$. In view of this circumstance, hereafter 
in our investigation of the Dyment-Muchnik lattice as a Brouwer algebra we 
shall work with mass problems of sets of numbers\footnote{A similar approach, 
which considers mass problems of sets rather than mass problems of total 
functions, has been widely used in the literature on the Medvedev lattice and 
the Muchnik lattice, the two ``Turing reducibility'' counterparts of 
$\mathscr{M}_D$ and $\mathscr{M}_{D,w}$.}, viewing $\mathscr{M}_{D,w}$ as 
$\mathscr{M}_{D,w}^{set}$. In this characterization of $\mathscr{M}_{D,w}$, 
the least element is the degree of any mass problem containing some c.e.\ 
set, the greatest element is the degree of the empty set, and the operations 
$\lor$, $\wedge$ and $\rightarrow$ are as follows: if $\mathcal{A}, 
\mathcal{B}$ are mass problems of sets, then $[\mathcal{A}]_{D,w}\lor 
[\mathcal{B}]_{D,w}= [\{A\oplus B: A \in \mathcal{A}, \, B \in 
\mathcal{B}]_{D,w}$, $[\mathcal{A}]_{D,w}\wedge [\mathcal{B}]_{D,w}= [ 
\mathcal{A} \cup \mathcal{B}]_{D,w}$, and $[\mathcal{A}]_{D,w}\rightarrow 
[\mathcal{B}]_{D,w}=[\{C: \QA{A \in \mathcal{A}}\QE{B\in \mathcal{B}}[B 
\leq_e C \oplus A]\}]_{D,w}$. On the other hand, we know by 
Theorem~\ref{thm:iso-usl-muchnik} that $\mathscr{M}_{D,w} \simeq 
\up(\mathscr{D}_e^{pf}) \simeq \up(\mathcal{P}, \leq_e\rangle)$. The same 
argument is easily seen to show that $\mathscr{M}_{D,w}^{set} \simeq 
\up(\mathscr{D}_e^{set}) \simeq \up(\langle 2^\omega, \leq_e \rangle)$. 
Indeed, the mapping which sends an upwards $\leq_e$-closed subset $A 
\subseteq \mathscr{D}^{pf}_e$ to $\{F(\mathbf{a}): \mathbf{a}\in A\}$ is an 
isomorphism of $\up(\mathscr{D}_e^{pf})$ onto $\up(\mathscr{D}_e^{set})$.  

\subsection{More on Brouwer algebras and intermediate logics}

As is well known, every partial order $\mathscr{P}$ set can be viewed as a 
\emph{Kripke frame}, and as such one can talk about $\Th_K(\mathscr{P})$, 
i.e. the set of propositional formulas that are valid in $\mathscr{P}$ 
according to Kripke semantics. We do not go into details here: see 
\cite{Chagrov-Zakharyashev, Gabbay-book} for suitable references. Suffice it 
to say that one can prove 

\begin{fact}\label{fact:DeJongh}
$\Th_K(\mathscr{P})=\Th(\up(\mathscr{P}))$. 
\end{fact} 

\begin{proof}
See for instance \cite[Theorem~7.20]{Chagrov-Zakharyashev}.
\end{proof}

Now, given partial orders $\langle P_1, \leq_1\rangle$, $\langle P_2,
\leq_2\rangle$, a function $f: P_1 \longrightarrow P_2$ is called a
\emph{$p$-morphism} (see \cite{deJongh-Troelstra}) if
\begin{itemize}
  \item $x \leq_1 y$ implies $f(x) \leq_2 f(y)$, for all $x,y \in P_1$;
  \item for every $x \in P_1$, and $y \in P_2$, if $f(x)\leq_2 y$ then
      there exists $z\in P_{1}$ such that $x\le_1 z$ and $f(z)=y$.
\end{itemize}
If there exists a $p$-morphism from $\mathscr{P}_1$ to $\mathscr{P}_2$ then
$\Th_K(\mathscr{P}_1) \subseteq \Th_K(\mathscr{P}_2)$: see
\cite[Corollary~2.17]{Chagrov-Zakharyashev}.

Furthermore, let $\langle 2^{<\omega}, \preceq \rangle$ be the poset where 
$2^{<\omega}$ is the set of $0,1$-valued finite strings and $\preceq$ is the 
partial ordering on strings so that $x \preceq y$ if $x$ is a prefix of $y$. 
It is known: 

\begin{lemma}\cite[Theorem~4.12]{Gabbay-book}
$\IPC=\Th_K\left(\langle 2^{<\omega}, \preceq \rangle\right)$.
\end{lemma}

The following definition comes from \cite{Kuyper-Muchnik}: in fact it is one 
of the equivalent characterizations proved by Kuyper of 
\cite[Definition~3.1]{Kuyper-Muchnik}. 

\begin{definition}\label{def:splittingclass}
Let $\mathscr{U}=\langle U, +, \le\rangle$ be an upper semilattice. A 
nonempty and countable downwards~$\leq$-closed $A\subseteq U$ is a 
\emph{splitting class} if and only if for every $a \in A$ and every finite 
$B\subseteq \{b \in A: b\nleq a\}$ there is $c \in A$, $c> a$, such that for 
all $b \in B$, $b+ c \notin A$. 
\end{definition}

If $\mathscr{U}=\langle U, +, \le\rangle$ is a pre-upper semilattice, then we 
say that a nonempty and countable downwards~$\leq$-closed $A\subseteq U$ is a 
splitting class in $\mathscr{U}$ if the set of $\equiv$-classes of the 
elements lying in $A$ form a splitting class in $\mathscr{U}_{/\equiv}$, see 
Remark~\ref{rem:iso-upsets}. Of course, in this case, $A$ is a splitting 
class in $\mathscr{U}$ if and only if $A_{/\equiv}$ is a splitting class in 
$\mathscr{U}_{/\equiv}$. 

The following results (along the lines of \cite{Kuyper-Muchnik}) constitute 
the final ingredients needed to reach our goal. 

\begin{lemma}\cite[Proposition~5.6]{Kuyper-Muchnik}\label{lem:binary-tree}
If $A$ is a splitting class in an upper semilattice $\mathscr{U}$ with a 
least element, then there exists a $p$-morphism $f: \langle A, \leq \rangle 
\longrightarrow \langle 2^{<\omega}, \preceq \rangle$. Hence 
\[
\Th_K\left(\langle A, \leq\rangle\right) \subseteq  
\Th_K\left(\langle 2^{<\omega}, \preceq \rangle\right)=\IPC.
\]
\end{lemma}

\begin{proof}
See \cite[Proposition~5.6]{Kuyper-Muchnik} where the proof is given when 
$\mathscr{U}$ is the upper semilattice of Turing degrees, but it works for 
every upper semilattice with a least element.
\end{proof}

From these lemmata we get:

\begin{theorem}\cite[Proposition~5.2 and Theorem~5.7]{Kuyper-Muchnik}\label{cor:fund-splitting} If $A$ is a splitting class in 
$\mathscr{D}_e^{set}$ and $\mathcal{A}=\{B\in 2^\omega: \deg_e(B) \in A\}$
then 
\[
\Th\left(\left[\mathbf{0}_{D,w}, 
[\mathcal{A}^c]_{D,w}\right]_{\mathscr{M}_{D,w}}\right)=\IPC,
\]
where $\mathcal{A}^c=2^\omega \smallsetminus \mathcal{A}$.
\end{theorem}

\begin{proof}
$\mathcal{A}$ is downwards $\leq_e$-closed in  $\langle 2^\omega, 
\leq_e\rangle$, so 
\[
\up\left(\langle \mathcal{A}, \leq_e \rangle\right) \simeq  
\left[2^\omega, \mathcal{A}^c\right]_{\up(\langle 
2^\omega, \leq_e\rangle)}
\]
via the assignment $B \mapsto B \cup \mathcal{A}^c$, for every $B \in \up(\langle 
\mathcal{A}, \leq_e \rangle)$. Since $A$ is a splitting class in
$\mathscr{D}_e^{set}$, then from Lemma~\ref{lem:binary-tree}, and the 
earlier observation in Fact~\ref{fact:DeJongh} that for every partial order 
$\mathscr{P}$ we have $\Th_K(\mathscr{P})=\Th(\up(\mathscr{P}))$, it follows 
that  
\[
\Th_K\left(\langle A, \leq_e \rangle\right)= \Th\left(\up(\langle A, 
\leq_e\rangle)\right)
= \Th\left(\up(\langle \mathcal{A}, \leq_e\rangle)\right)
= \Th\left(\left[2^\omega, \mathcal{A}^c\right]_{\up(\langle 2^\omega, 
\leq_e\rangle)}\right) \subseteq \IPC.
\]
But, $\mathscr{M}_{D,w} \simeq \up(\langle 2^\omega, \leq_e\rangle)$ and, 
under this isomorphism, $2^\omega \mapsto \mathbf{0}_{D,w}$ and 
$\mathcal{A}^c \mapsto [\mathcal{A}^c]_{D,w}$. Hence 
\[
\Th\left(\left[\mathbf{0}_{D,w}, [\mathcal{A}^c]_{D,w}\right]_{\mathscr{M}_{D,w}}
\right)=\IPC.  
\qedhere
\]
\end{proof}

We have just proved that if there exists a splitting class of $e$-degrees, 
then there exists an initial segment of $\mathscr{M}_{D,w}$ which models 
exactly $\IPC$. It is only left to show that splitting classes of $e$-degrees 
do exist. This will be done in the next subsection. 
 
\subsection{A splitting class in the $e$-degrees}

We will show that every countable downwards~$\leq_e$-closed class 
$\mathcal{I}$ of sets can be extended to a countable splitting class. We will 
do this by inductively adding to $\mathcal{I}$ more and more sets, eventually 
getting a countable downwards~$\leq_e$-closed set $\mathcal{I}_f$, with 
enough added elements to witness that for every pair $(A, \mathcal{B})$ where 
$A\in \mathcal{I}_f$ and $\mathcal{B} \subseteq \mathcal{I}_f$ is a finite 
set such that $B \nleq_e A$ for every $B \in \mathcal{B}$, there exists $C 
\in \mathcal{I}_f$ with $C>_e A$ and $C \oplus B \notin \mathcal{I}_f$ for 
every $B \in \mathcal{B}$. The following lemma helps describe the inductive 
steps in the construction of the splitting class. 

Throughout this subsection we will treat a given set $\Gamma$ of numbers as a 
\emph{generalized enumeration operator}, i.e. as a mapping from $2^\omega$ to 
$2^\omega$, such that 
\[
\Gamma(A)=\{x: \QE{u}[\langle x,u\rangle \in \Gamma \,\&\, D_u \subseteq A]\}.
\]
If $\mathcal{A} \subseteq 2^\omega$ then, again, we denote $\mathcal{A}^c= 
2^\omega \smallsetminus \mathcal{A}$.
 
\begin{lemma}\label{lem:a-splitting}
Suppose that $\mathcal{I},\mathcal{G}$ are sets of sets of numbers and 
that $A, B$ are sets such that 

\begin{enumerate}
  \item[(a)] $\mathcal{I}$ is countable and downwards~$\leq_e$-closed;
  \item[(b)] $\mathcal{G} \subseteq \mathcal{I}^c$ is countable and closed 
      under $\equiv_e$; 
  \item[(c)] $A, B \in \mathcal{I}$ and $B \nleq_e A$. 
\end{enumerate} 

Then there exists a generalized enumeration operator $\Gamma$ satisfying 

\begin{enumerate}
  \item[(0)] $\Gamma \oplus A \nleq_e A$
  \item if $X \in \mathcal{I}$ then $\Gamma \oplus A\geq_e X \textrm{ if 
      and only if } A \geq_e X$; 
  \item if $X \in \mathcal{G}$ then $\Gamma \oplus A \ngeq_e X$; 
  \item $\Gamma(B) \notin \mathcal{I}$;
  \item $\Gamma\oplus A \ngeq_e \Gamma(B)$.
\end{enumerate}
\end{lemma}

\begin{proof}
Obtain $\Gamma$ by forcing with $\langle P, \leq_P \rangle$ where 
\begin{enumerate}
  \item[I.] each condition $p \in P$ is a triple $p=\langle \Gamma_p, B_p, 
      N_p\rangle$, where 
      
\begin{enumerate}
  \item[-] $\Gamma_p$ is a finite set of axioms $\langle n, u\rangle$ 
      where $n\in \omega$ and $u$ is the canonical index of the finite 
      set $D_u$; 
  \item[-] $B_p$ is a finite subset of $B$;
  \item[-] $N_p$ is a finite subset of $\omega$.
\end{enumerate}
  \item[II.] $p\leq_P q$ if
  \begin{enumerate}
    \item[-] $\Gamma_q \subseteq \Gamma_p$ (each axiom of $\Gamma_q$ is 
        an axiom of $\Gamma_p$);
    \item[-] $B_q \subseteq B_p$ and $N_q \subseteq N_p$;
    \item[-] $\QA{\langle n, u\rangle \in \Gamma_p \smallsetminus 
        \Gamma_q} \left[ B_p \subseteq D_u \right]$ (the new axioms must 
        use $B_p$); 
    \item[-] $\QA{\langle n, u\rangle \in \Gamma_p \smallsetminus 
        \Gamma_q} \left[ n \notin N_p \right]$ (no new axiom enumerates 
        elements of $N_p$). 
  \end{enumerate}
\end{enumerate} 

We will exhibit appropriate dense sets, so that if $\Gamma$ comes from a 
sufficiently generic filter then it satisfies the claims of the lemma.

The following sublemma addresses item (0) of Lemma~\ref{lem:a-splitting}. 

\begin{sublemma}\label{lem:0}
$\Gamma \nleq_e A$.
\end{sublemma}

\begin{proof}[Proof~of~Sublemma~\ref{lem:0}]
Let $\Theta$ be an $e$-operator. The set of conditions 
\[ 
\mathcal{C}=\{p: \QE{m}\left[m \in \Theta(A) \,\&\, p \Vdash m 
\notin \Gamma\right]\} \cup
\{\QE{m}\left[m \notin \Theta(A) \,\&\, p \Vdash m \in \Gamma\right]\}
\]
is dense. To see this, suppose $p=\langle \Gamma_p, B_p, N_p\rangle$ is a 
condition. Let $m=\langle m_0, m_1\rangle$ be such that $m, m_0, m_1$ are 
larger than any number mentioned in $p$, with $D_{m_1} \supseteq B_p$.
If $m \in \Theta(A)$, then let
$q=\langle \Gamma_p, B_p, N_p\cup \{m_0\}\rangle$
so that $q\Vdash \langle m_0, m_1\rangle \notin \Gamma$.
If $m \notin \Theta(A)$, then let $q=\langle \Gamma_p \cup \{m\}, B_p, N_p\rangle$.
In either case there is a condition $q \leq_P p$ such that $q \in \mathcal{C}$. 
\end{proof}

The following sublemma addresses items (1) and (2) of 
Lemma~\ref{lem:a-splitting}. 
 
\begin{sublemma}\label{lem:1}
For $X \in \mathcal{I}$,  $\Gamma \oplus A\geq_e X $ if and only if $A \geq_e 
X$. For $X \in \mathcal{G}$, $\Gamma \oplus A \ngeq_e X$. 
\end{sublemma}

\begin{proof}[Proof~of~Sublemma~\ref{lem:1}] 
Let $X \in \mathcal{I} \cup \mathcal{G}$. Suppose that $\Theta$ is an $e$-operator and 
that $\Theta(\Gamma\oplus A)=X$. This implies that the set of conditions 
\begin{align*}
\mathcal{C} &=\{p: p \Vdash \Theta(\Gamma\oplus A)\ne X\}\\
            &=\{p: \QE{m}\left [m \notin X \,\&\, p \Vdash m \in \Theta(\Gamma 
      \oplus A)\right]\} \cup \{p: \QE{m} \left[ m \in X \,\&\, p \Vdash m \notin 
      \Theta(\Gamma \oplus A) \right]\} 
\end{align*}
is not dense. (Here, if $p=\langle \Gamma_p, B_p, N_p \rangle$ is a 
condition, then $p \Vdash m \in \Theta(\Gamma \oplus A)$ means $m \in 
\Theta(\Gamma_p \oplus A)$, and $p \Vdash m \notin \Theta(\Gamma \oplus A)$ 
means that $\QA{q \le_{P} p}[m \notin \Theta(\Gamma_q \oplus A)]$. Analogous 
definitions will be tacitly assumed throughout the proof of 
Lemma~\ref{lem:a-splitting}.) Hence there exists a condition $p$ with no 
extension in this set, which implies 
\[
\QA{m} \big[ m \in X \Leftrightarrow \QE{q\leq_P p} 
[m \in \Theta(\Gamma_q \oplus A) ] \big].
\] 
Then $A\geq_e X$: indeed, $A$ enumerates $m \in X$ by searching for a 
condition $q \leq_P p$ forcing $m \in \Theta(\Gamma_q\oplus A)$. The set of 
such conditions $q$ is positively c.e. in $A$. So we conclude that 
\[
\QA{X \in \mathcal{I} \cup \mathcal{G}}[ \Gamma \oplus A \geq_e
X \Leftrightarrow A \geq_e X]. \qedhere
\] 
\end{proof}

The next sublemma addresses item (3) of Lemma~\ref{lem:a-splitting}. 

\begin{sublemma}\label{lem:2}
$\Gamma(B) \notin \mathcal{I}$. 
\end{sublemma}

\begin{proof}[Proof~of~Sublemma~\ref{lem:2}] 
For every condition $p=\langle \Gamma_p, B_p, N_p\rangle$ and $X \in 
\mathcal{I}$, condition $p$ has an extension in the set $\mathcal{C}$ of 
conditions 
\[
\mathcal{C}=\{q: \QE{m} \left[m \in X \,\&\, 
q \Vdash m \notin \Gamma(B)\right]\}
\cup \{q: \QE{m} \left[m \notin X \,\&\, q \Vdash m \in \Gamma(B)\right]\},
\]
which therefore is dense, and thus $X \ne \Gamma(B)$. Indeed, to see density, 
let $p=\langle \Gamma_p, B_p, N_p\rangle$ be a condition, and let $m$ be 
larger than any number mentioned in $p$. If $m\in X$, then the condition 
$q=\langle \Gamma_p, B_p, N_p \cup \{m\} \rangle$ extends $p$ and forces $m 
\notin \Gamma(B)$. If $m\notin X$, then the condition $q=\langle \Gamma_p\cup 
\{\langle m, u\rangle\}, B_p, N_p  \rangle$ where $D_u = B_p$ extends $p$ with $q \Vdash m \in 
\Gamma(B)$. 
\end{proof}

The next sublemma addresses item (4) of Lemma~\ref{lem:a-splitting}. 

\begin{sublemma}\label{lem:3}
$\Gamma\oplus A \ngeq_e \Gamma(B)$.
\end{sublemma}

\begin{proof}[Proof~of~Sublemma~\ref{lem:3}] 
Suppose that $p=\langle \Gamma_p, B_p, N_p\rangle$ is a condition, and let 
$\Theta$ be an $e$-operator. Let $\mathcal{C}$ be the set of conditions 
\begin{multline*}
\mathcal{C}=\{q: \QE{m} \left[ q \Vdash m \in \Theta(\Gamma \oplus A) \,\&\,
q \Vdash m \notin \Gamma(B)\right]\},\\
\cup \{q: \QE{m} \left[q \Vdash m \in \Gamma(B) \,\&\, q \Vdash m \notin 
\Theta(\Gamma \oplus A) \right]\}.
\end{multline*}
In Case~1 and Case~2 below, $m$ and $d$ vary through the numbers greater than 
any number mentioned in $p$. 

\underline{Case~1.} $\QE{m}\QE{d \notin B}\QE{q \leq_P p}
\bigl[m \in \Theta(\Gamma_q\oplus A) \,\&\, \QA{\langle n, u 
\rangle \in \Gamma_q \smallsetminus \Gamma_p}[d \in D_u]\bigr]$.

Let
\[
r=\langle \Gamma_q, B_q,  N_q \cup\{m\}\rangle.
\]
Then $r <_P p$, $r\Vdash   m \notin \Gamma(B)$ since $\Gamma_p$ does not 
mention $m$ and the axioms in $\Gamma_q \smallsetminus \Gamma_p$ do not apply 
to $B$, and $ r \Vdash m \in \Theta(\Gamma \oplus A)$. Hence $p$ has an 
extension $r \in \mathcal{C}$. 

\underline{Case~2.} $\QE{m}\QE{d \in B}[ \langle \Gamma_p, B_p\cup \{d\}, 
N_p\rangle \Vdash m \notin \Theta(\Gamma \oplus A)]$ (thus requiring that 
using $d$ in all new axioms makes it impossible for $\Theta(\Gamma \oplus A)$ 
to enumerate $m$). 

Let $r=\langle \Gamma_p \cup \{\langle m, u \rangle\}, 
B_p\cup\{d\}, N_p \rangle$, where $D_u = B_p\cup \{d\}$. Then $r<_P p$, 
$r\Vdash m \notin \Theta(\Gamma 
\oplus A)$, and $r \Vdash m \in \Gamma(B)$. 

Again, $p$ has an extension $r \in \mathcal{C}$.

\underline{Case~3.} Otherwise. 

Then, for all sufficiently large $d$, 
\begin{equation}\label{eq:case3}
d \in B \Leftrightarrow\\
\QE{m}\QE{q \leq_P p}
\big[m \in \Theta (\Gamma_q \oplus A)\;\&\; \QA{\langle n, u\rangle \in 
\Gamma_q \smallsetminus \Gamma_p} 
    [d \in D_u]\big].
\end{equation}
Indeed, the righthand side of (\ref{eq:case3}) implies that $d\in B$, since 
otherwise Case~1 would hold. On the other hand, if $d \in B$, then 
$\langle\Gamma_p, B_p\cup \{d\}, N_p\rangle \nVdash m \notin \Theta(\Gamma 
\oplus A)$ for every sufficiently large $m$ since Case~2 does not hold. This 
means that there is a $q \leq_P \langle\Gamma_p, B_p\cup \{d\}, N_p\rangle 
\leq_P p$ with $m \in \Theta(\Gamma_q \oplus A)$, and $\QA{\langle n, 
u\rangle \in \Gamma_q \smallsetminus \Gamma_p}[d \in D_u]$ by the definition 
of extension.  Thus the righthand side of (\ref{eq:case3}) holds. 

But then, contradicting the assumption, it follows that $B\leq_e A$ since the 
description in the righthand side of (\ref{eq:case3}) is positive in $A$. 

It follows that every $p$ has an extension in $\mathcal{C}$, so $\mathcal{C}$ 
is a dense set of conditions. 
\end{proof}

In conclusion, if $\Gamma$ is sufficiently generic and intersects the dense 
sets of conditions exhibited in the proof then the generalized enumeration 
operator $\Gamma$ has the desired properties, which satisfy 
Lemma~\ref{lem:a-splitting}. 
\end{proof}

\begin{corollary}\label{cor:splitting1}
Let $\mathcal{I},$ $\mathcal{G}$, $A,B$ be as in the hypotheses of 
Lemma~\ref{lem:a-splitting}, and let $\Gamma$ be the generalized enumeration 
operator built in the proof of the lemma. Define $\mathcal{I}_+= \mathcal{I} 
\cup \{X: X \leq_e C\}$ where $C=\Gamma \oplus A$, and define $\mathcal{G}_+= 
\mathcal{G} \cup \{X: X \equiv_e \Gamma(B)\}$. Then the following hold: 
$\mathcal{I}_+$ and $\mathcal{G}_+$ are countable, $\mathcal{I}_+ \supseteq 
\mathcal{I}$, $\mathcal{G}_+ \supseteq \mathcal{G}$, $\mathcal{I}_+$ is 
downwards~$\leq_e$-closed, $\mathcal{G}_+$ is closed under $\equiv_e$,  
$\mathcal{G}_+ \subseteq (\mathcal{I}_+)^c$, $\Gamma(B) \in \mathcal{G}_+$ 
(hence $\Gamma(B) \notin \mathcal{I}_+$). Finally, $C  >_e A$, $C \in 
\mathcal{I}_+$ (hence $C \notin \mathcal{G}_+$), and $B\oplus C \notin  
\mathcal{I}_+$. 
\end{corollary}

\begin{proof}
By construction $\mathcal{I}_+$ and $\mathcal{G}_+$ are countable, 
$\mathcal{I}_+$ is downwards~$\leq_e$-closed, and $\mathcal{G}_+$ is closed 
under $\equiv_e$. To show that $\mathcal{G}_+ \subseteq (\mathcal{I}_+)^c$, it 
is enough to show that $\Gamma(B)\notin \mathcal{I}_+$: this follows from the 
fact that $\Gamma(B)\notin \mathcal{I}$ by Lemma~\ref{lem:a-splitting}(3), 
and $\Gamma(B) \nleq_e \Gamma \oplus A$ by Lemma~\ref{lem:a-splitting}(4). By 
Lemma~\ref{lem:a-splitting}(0) we have that $C >_e A$.  By construction we 
have $C \in \mathcal{I}_+$. Finally, notice that $\Gamma(B) \leq_e B\oplus 
C$: thus, if $B\oplus C \in  \mathcal{I}_+$ then $\Gamma(B)\in \mathcal{I}$ 
or $\Gamma(B)\leq_e \Gamma \oplus A$, but both of these two cases have
already been excluded. In fact, $B \oplus C \leq_e \Gamma \oplus A$ can be 
excluded also for the fact that it would give $B\leq_e \Gamma\oplus A$ and 
thus, by Sublemma~\ref{lem:1}, $B \leq_e A$. 
\end{proof} 

\begin{remark}\label{rem:can-look}
We can look at the construction described in the proof of 
Lemma~\ref{lem:a-splitting} as the partial assignment $s_+$, defined on all 
quadruples $(\mathcal{I}, \mathcal{G}, A, B)$ as in the hypotheses of 
Lemma~\ref{lem:a-splitting}, such that
\[
s_+(\mathcal{I}, \mathcal{G}, A, B)= (\mathcal{I}_+, \mathcal{G}_+,
\Gamma),
\] 
where
\begin{itemize}
  \item $\mathcal{I}_+$ is the set constructed in 
      Corollary~\ref{cor:splitting1} starting from $\mathcal{I}, 
      \mathcal{G}, A, B$; 
  \item $\mathcal{G}_+$ is the set constructed in 
      Corollary~\ref{cor:splitting1} starting from $\mathcal{I}, 
      \mathcal{G}, A, B$; 
  \item $\Gamma$ is the generalized enumeration operator constructed in 
      Corollary~\ref{cor:splitting1} from $\mathcal{I}, \mathcal{G}, A, B$. 
\end{itemize} 
\end{remark}

\begin{lemma}\label{lem:a-splitting-2}
Suppose that $\mathcal{I},\mathcal{G}$ are sets of sets of numbers, $A$ is a
set, and $\mathcal{B}=\{B_1, \ldots, B_k\}$ is a finite set of sets such that 

\begin{enumerate}
  \item[(a)] $\mathcal{I}$ is countable and downwards~$\leq_e$-closed;
  \item[(b)] $\mathcal{G} \subseteq \mathcal{I}^c$ is countable and closed 
      under $\equiv_e$; 
  \item[(c)] $A \in \mathcal{I}$ and, for every $1\leq  i \leq k$, $B_i \in 
      \mathcal{I}$ and $B_i \nleq_e A$. 
\end{enumerate}
Then there exist sets $\mathcal{I}^+$, $\mathcal{G}^+$, and a set $C$
such that $\mathcal{I}^+$ and 
$\mathcal{G}^+$ are countable, $\mathcal{I}^+ \supseteq \mathcal{I}$, 
$\mathcal{G}^+ \supseteq \mathcal{G}$, $\mathcal{I}^+$ is downwards~ 
$\leq_e$-closed, $\mathcal{G}^+$ is closed under $\equiv_e$,  $\mathcal{G}^+ 
\subseteq (\mathcal{I}_+)^c$, $C 
>_e A$, $C \in \mathcal{I}^+$ (hence $C \notin \mathcal{G}^+$), 
and for every $1\leq i \leq k$, there is a $D \leq_e B_i \oplus C$
with $D \in \mathcal{G}^+$ (hence $D \notin \mathcal{I}^+$).
\end{lemma}

\begin{proof}
Let us start with $\mathcal{I},\mathcal{G}$ and a pair $(A, \mathcal{B})$ as 
in the statement of this lemma. Define a sequence 
$\{(\mathcal{I}_i, \mathcal{G}_i, \Gamma_i): 1 \leq i \leq k\}$ as follows, 
where $s_+$ is the partial function defined in Remark~\ref{rem:can-look}:
\begin{itemize}
  \item $(\mathcal{I}_1, \mathcal{G}_1, \Gamma_1) = 
      s_+(\mathcal{I}_0,\mathcal{G}_0, A, B_1)$ 
      
  \item $(\mathcal{I}_{i+1}, \mathcal{G}_{i+1}, \Gamma_{i+1}) = 
      s_+(\mathcal{I}_i, \mathcal{G}_i, \hat{\Gamma}_i \oplus A, B_{i+1})$ 
      for $1 \leq i < k$,
\end{itemize}

where we understand that $\mathcal{I}_0=\mathcal{I}$, 
$\mathcal{G}_0=\mathcal{G}$,  $\hat{\Gamma}_1 = \Gamma$, and 
$\hat{\Gamma}_{i+1}=\Gamma_{i+1}\oplus \hat{\Gamma}_i$. 

The following claims are proved by induction on $i$ with $1 \le i \leq k$, 
using Lemma~\ref{lem:a-splitting} and Corollary~\ref{cor:splitting1}. 

\begin{enumerate}
  \item[(0)] $\hat{\Gamma}_i \oplus A \nleq_e A$ if $i=1$, and 
      $\hat{\Gamma}_i \oplus A \nleq_e \hat{\Gamma}_{i-1} \oplus A$ if 
      $i>1$; 
      
  \item if $X \in \mathcal{I}_{i-1}$, then $\hat{\Gamma}_i \oplus A \geq_e 
      X$ if and only if $\hat{\Gamma}_{i-1} \oplus A \geq_e X$; 
      
  \item if $X \in \mathcal{G}_{i-1}$, then $\hat{\Gamma}_i \oplus A \ngeq_e 
      X$; 
  
  \item $\Gamma_i(B_i) \notin \mathcal{I}_{i-1}$;
  
  \item $\hat{\Gamma}_i \oplus A \ngeq_e \Gamma_i(B_i)$.
\end{enumerate}

Moreover, it immediately follows from the construction and 
Corollary~\ref{cor:splitting1} that for every $1\leq i \leq k$, we have 
$\mathcal{I}_{i-1} \subseteq \mathcal{I}_{i}$, $\mathcal{G}_{i-1} \subseteq 
\mathcal{G}_{i}$, $\hat{\Gamma}_i \oplus A \in \mathcal{I}_i$,
$\Gamma_i(B_i) \in \mathcal{G}_i$, and for every $i \leq 
k$, $\mathcal{G}_{i} \subseteq (\mathcal{I}_i)^c$. 

Define $C = \hat{\Gamma}_k \oplus A$, $\mathcal{I}^+= \mathcal{I}_k$, and 
$\mathcal{G}^+= \mathcal{G}_k$. For each $1 \leq i \leq k$, define $D_i = 
\Gamma_i(B_i)$. Then $\mathcal{I}^+$ and $\mathcal{G}^+$ are countable, 
extend $\mathcal{I}$ and $\mathcal{G}$ respectively, $\mathcal{I}^+$ is 
downwards~$\leq_e$-closed, $\mathcal{G}^+$ is closed under $\equiv_e$,  
$\mathcal{G}^+ \subseteq (\mathcal{I^+})^c$, $C  >_e A$, and $C \in 
\mathcal{I}^+$.  Furthermore, for each $1 \leq i \leq k$, $D_i \in 
\mathcal{G}_i \subseteq \mathcal{G}^+$ and $D_i \leq_e B_i \oplus \Gamma_i 
\leq_e B_i \oplus C$. 
\end{proof}

\begin{remark}\label{rem:final}
We can look at the construction described in the proof of 
Lemma~\ref{lem:a-splitting-2} as the (partial) assignment 
\[
s_{++}(\mathcal{I}, \mathcal{G}, A, \mathcal{B})= 
(\mathcal{I}^+, \mathcal{G}^+, C)
\] 
where $\mathcal{I}^+$, $\mathcal{G}^+$, and $C$ are as in the
conclusion of Lemma~\ref{lem:a-splitting-2} when starting from
$\mathcal{I}$, $\mathcal{G}$, $A$, and $\mathcal{B}$.
\end{remark}

Finally, 
\begin{theorem}\label{thm:final}
For any pair  $\mathcal{I},\mathcal{G}$ of sets of numbers such that 
\begin{enumerate}
  \item[(a)] $\mathcal{I}$ is countable and downwards~$\leq_e$-closed;
  \item[(b)] $\mathcal{G} \subseteq \mathcal{I}^c$ is countable and closed 
      under $\equiv_e$,
\end{enumerate} 
there exists a pair of countable sets $(\mathcal{I}_f, \mathcal{G}_f)$ such 
that $(\mathcal{I}_f, \mathcal{G}_f)\supseteq (\mathcal{I}, \mathcal{G})$ 
(meaning that $\mathcal{I}_f \supseteq \mathcal{I}$ and $\mathcal{G}_f 
\supseteq \mathcal{G}$), $\mathcal{I}_f$ is downwards~$\leq_e$-closed, 
$\mathcal{G}_f$ is closed under $\equiv_e$, $\mathcal{G}_f \subseteq 
(\mathcal{I}_f)^c$, and $\mathcal{I}_f$ forms a splitting class. 
\end{theorem}  

\begin{proof}
Let $\mathcal{I},\mathcal{G}$ be as in the statement of the theorem. We 
define by transfinite induction on $\beta \leq \omega^2$, a sequence of pairs 
$(\mathcal{I}_\beta, \mathcal{G}_\beta)$ such that $(\mathcal{I}_\beta, 
\mathcal{G}_\beta)\subseteq (\mathcal{I}_\gamma, \mathcal{G}_\gamma)$ for 
every $\beta < \gamma \leq \omega^2$. Our desired final pair $(\mathcal{I}_f, 
\mathcal{G}_f)$ will be $(\mathcal{I}_{\omega^2}, \mathcal{G}_{\omega^2})$. 

\medskip
\paragraph{\emph{$\beta=0$.}}
Let $(\mathcal{I}_0, \mathcal{G}_0)= (\mathcal{I}, \mathcal{G})$.

\medskip
\paragraph{\emph{$\beta$ successor.}} 
Suppose that $\beta=\omega\cdot j +i$ where $i,j$ are natural numbers such 
that $j \ge 0$ and $i>0$. Let $\{(A_r^{\omega \cdot j}, \mathcal{B}_r^{\omega 
\cdot j}): r>0\}$ be an enumeration of all pairs such that $A_r^{\omega \cdot 
j}\in \mathcal{I}_{\omega\cdot j}$ and $\mathcal{B}_r^{\omega \cdot j}$ is a 
finite subset of $\mathcal{I}_{\omega\cdot j}$ such that $B \nleq_e 
A_r^{\omega \cdot j}$ for all $B \in \mathcal{B}_r^{\omega \cdot j}$. 

Let $s_{++}$ be the partial function defined in Remark~\ref{rem:final}, and 
suppose 
\[
s_{++}(\mathcal{I}_{\omega\cdot j+i-1}, \mathcal{G}_{\omega\cdot j+i-1}, 
A_i^{\omega\cdot j},  \mathcal{B}_i^{\omega\cdot j}) 
= (\mathcal{I}^+, \mathcal{G}^+, C).
\]
Define $\mathcal{I}_{\beta} = \mathcal{I}^+$ and $\mathcal{G}_{\beta} = 
\mathcal{G}^+$, and also record that $C_\beta = C$.

\medskip
\paragraph{\emph{$\beta$ limit.}}
Define $\mathcal{I}_{\beta}= \bigcup_{\gamma < \beta} \mathcal{I}_{\gamma}$ 
and $\mathcal{G}_{\beta}= \bigcup_{\gamma < \beta} \mathcal{G}_{\gamma}$. 

\medskip

It is easy to see that each of $\mathcal{I}_f$ and $\mathcal{G}_f$ is union 
of a chain of sets. Each of the addenda in the union yielding $\mathcal{I}_f$ 
is downwards~$\le_e$-closed, and so is $\mathcal{I}_f$. Each of the addenda 
in the union yielding $\mathcal{G}_f$ is $\equiv_e$-closed, and so is 
$\mathcal{G}_f$. 

In order to show that $\mathcal{G}_f \subseteq (\mathcal{I}_f)^c$, let us 
prove by ordinal induction on $\beta \leq \omega^2$ that $\mathcal{G}_\beta 
\subseteq (\mathcal{I}_\beta)^c$. The claim is straightforward if $\beta=0$ 
or $\beta$ is limit. So, suppose that $\beta$ is a successor and has the form 
$\beta=\omega \cdot j+i$ with $j \ge 0$ and $i>0$. We are assuming that 
$\mathcal{I}_{\omega\cdot j}$, $\mathcal{G}_{\omega\cdot j}$, 
$A_i^{\omega\cdot j}$, and $\mathcal{B}_i^{\omega\cdot j}$ satisfy the 
hypotheses of Lemma~\ref{lem:a-splitting-2}, so $\mathcal{I}_{\omega \cdot 
j+i-1}$, $\mathcal{G}_{\omega \cdot j+i-1}$, $A_i^{\omega\cdot j}$, and 
$\mathcal{B}_i^{\omega\cdot j}$ also satisfy these hypotheses because 
$\mathcal{I}_{\omega \cdot j} \subseteq \mathcal{I}_{\omega \cdot j+i-1}$ and 
$\mathcal{G}_{\omega \cdot j+i-1} \subseteq (\mathcal{I}_{\omega \cdot 
j+i-1})^c$. Thus $s_{++}(\mathcal{I}_{\omega\cdot j+i-1}, 
\mathcal{G}_{\omega\cdot j+i-1}, A_i^{\omega\cdot j},  
\mathcal{B}_i^{\omega\cdot j})$ produces $\mathcal{I}_\beta$ and 
$\mathcal{G}_\beta$ with $\mathcal{G}_\beta \subseteq (\mathcal{I}_\beta)^c$. 

Finally, let us show that $\mathcal{I}_f$ is a splitting class. Suppose that 
$A \in \mathcal{I}_f$ and that $\mathcal{B}=\{B_1, \ldots, B_k\} \subseteq 
\mathcal{I}_f$ is a finite set such that $B_r \nleq_e A$ for every $1\le r 
\le k$. Suppose that $j \geq 0$ is least such that $A \in \mathcal{I}_{\omega 
\cdot j}$ and $\mathcal{B} \subseteq \mathcal{I}_{\omega \cdot j}$. Hence, 
for some $i >0$, $A=A_i^{\omega \cdot j}$ and 
$\mathcal{B}=\mathcal{B}_i^{\omega \cdot j}$. At stage $\beta = {\omega \cdot 
j} + i$, we produced $(\mathcal{I}_\beta, \mathcal{G}_\beta, C_\beta) = 
s_{++}(\mathcal{I}_{\beta-1}, \mathcal{G}_{\beta-1}, A_i^{\omega \cdot j}, 
\mathcal{B}_i^{\omega \cdot j})$ via Lemma~\ref{lem:a-splitting-2}.  Thus 
$C_\beta  >_e A_i^{\omega \cdot j}$ and $C_\beta \in \mathcal{I}_\beta 
\subseteq \mathcal{I}_f$.  Furthermore, $B_r \oplus C_\beta \notin 
\mathcal{I}_f$ for every $1 \leq r \leq k$.  This is because $B_r \oplus 
C_\beta \geq_e D$ for some $D \in \mathcal{G}_\beta \subseteq \mathcal{G}_f$. 
So if $B_r \oplus C_\beta \in \mathcal{I}_f$, then $D$ would be in 
$\mathcal{I}_f$ too because $\mathcal{I}_f$ is downwards~$\le_e$-closed, 
which would contradict that $\mathcal{G}_f \subseteq (\mathcal{I}_f)^c$. Thus 
$C_\beta$ is the required witness for $A$ and $\mathcal{B}$. 
\end{proof}

\subsection{More on splitting classes in the $e$-degrees}\label{ssct:c}

Kuyper in \cite{Kuyper-Muchnik} exhibited several ``natural'' classes of sets 
which form splitting classes in the Turing degrees, including the low  Turing 
degrees, the $1$-generic Turing degrees below the halting set $K$, and the 
Turing degrees of $low_2$ hyperimmune-free functions. He also generalized the 
notion to splitting classes of cardinality $\aleph_1$ (see 
\cite[Definition~6.2]{Kuyper-Muchnik}) for which 
Theorem~\ref{cor:fund-splitting} is still valid, and he showed that under CH 
the full class of Turing degrees of hyperimmune-free functions form an 
$\aleph_1$ splitting class.

We have exhibited a splitting class in the $e$-degrees, but one might object 
that this is, after all, only an ad hoc example. On the other hand, we have 
tested several natural downward $\leq_e$-closed classes of sets, but 
unfortunately they all turned out not to be splitting classes because of 
properties of $\mathcal{K}$-pairs. To illustrate these unsuccessful attempts, 
we need to recall some basic facts about $\mathcal{K}$-pairs in the 
$e$-degrees. 

A pair $\{B_1, B_2\}$ of sets form a \emph{$\mathcal{K}$-pair relative  to a 
set $C$} if there exists a \emph{$\mathcal{K}$-pair-defining} set $M\leq_e C$ 
for the pair, meaning that $B_1 \times B_2 \subseteq M$ and for every 
$\langle x,y\rangle \in M$, either $x \in B_1$ or $y \in B_2$. A 
$\mathcal{K}$-pair relative to $C$ is \emph{nontrivial} if $B_1 \nleq_e C$ 
and $B_2\nleq_e C$. Clearly if $C\leq_e D$, then $\{B_1, B_2\}$ forms also a 
$\mathcal{K}$-pair relative to $D$. A pair $\{B_1, B_2\}$ is a 
\emph{$\mathcal{K}$-pair} if it is a $\mathcal{K}$-pair relative to 
$\emptyset$. Consequently, a $\mathcal{K}$-pair is nontrivial if neither of 
the two sets is c.e. 

In the following, for a given set $X$, we denote $K_X=\{e: e \in \Phi_e(X)\}$ 
and $J_e(X)= K_X \oplus (K_X)^c$. The set $J_e(X)$ is the \emph{enumeration 
jump} of $X$. A set $X$ is \emph{$e$-low} if $J_e(X)\equiv_e J_e(\emptyset)$, 
or equivalently $J_e(X)\equiv_e K^c$. A set $A$ is \emph{quasi-minimal} if 
$A$ is not c.e., and $f \leq_e A$ implies $f$ is computable. 

The following lemma collects several facts on $\mathcal{K}$-pairs, due to 
Kalimullin~\cite{Kalimullin}, used in the next four corollaries. 

\begin{lemma}\label{lem:Kprinc}
Let $\{B_1, B_2\}$ form a $\mathcal{K}$-pair relative to a set $C$. Then 
\begin{enumerate}
  
  \item $\{B_1\oplus C, B_2 \oplus C\}$ form a $\mathcal{K}$-pair relative 
      to $C$; 
  
  \item If $X_1 \leq_e B_1$ and $X_2 \leq_e B_2$, then $\{X_1, X_2\}$ form 
      a $\mathcal{K}$-pair relative to $C$; 
  
  \item If $B_2 \nleq_e C$, then $B_1 \leq_e C \oplus B_2^c$ and  $B_1^c 
      \le_e J_e(C) \oplus B_2$; 
  
  \item If $\{B_1, B_2\}$ is a nontrivial $\mathcal{K}$-pair relative to 
      $C$, then the $e$-degree of $C$ is the meet of the $e$-degrees of $B_1 
      \oplus C$ and $B_2 \oplus C$. (In fact, for every $X\geq_e C$, the 
      $e$-degree of $X$ is the meet of the $e$-degrees of $B_1 \oplus X$ 
      and $B_2 \oplus X$.)  
      
  \item If $\{B_1, B_2\}$ is a nontrivial $\mathcal{K}$-pair, then both 
      $B_1$ and $B_2$ are quasi-minimal.  
\end{enumerate}
\end{lemma}

\begin{proof}
See \cite{Kalimullin}.
\end{proof}

\begin{corollary}\label{cor:Kprinc}
Let $\{B_1, B_2\}$ form a $\mathcal{K}$-pair below $K^c$. Then 
\[
\QA{C \ngeq_e B_2}[J_e(B_1 \oplus C)\equiv_e J_e(C)].
\] 
\end{corollary}

\begin{proof}
Let $\{B_1, B_2\}$ be a $\mathcal{K}$-pair such that $B_1, B_2 \leq_e K^c$. 
Then $\{B_1, B_2\}$ is also a $\mathcal{K}$-pair relative to any set $C$. 
Hence by Lemma~\ref{lem:Kprinc}(1), $\{B_1\oplus C, B_2 \oplus C\}$ is a 
$\mathcal{K}$-pair relative to $C$, and thus by Lemma~\ref{lem:Kprinc}(2)  
$\{K_{B_1\oplus C}, B_2\}$ is a $\mathcal{K}$-pair relative to $C$. If $C 
\ngeq_e B_2$, by Lemma~\ref{lem:Kprinc}(3) applied to this 
$\mathcal{K}$-pair, we have that $\left( K_{B_1 \oplus C}\right)^c \leq_e 
J_e(C) \oplus B_2$. But $B_2  \leq J_e(\emptyset) \leq_e J_e(C)$, so $\left( 
K_{B_1 \oplus C}\right)^c \leq_e J_e(C)$. On the other hand, $K_{B_1 \oplus 
C} \leq_e B_1 \oplus C \leq_e B_1 \oplus J_e(C) \leq_e J_e(C)$ since $B_1  
\leq_e J_e(\emptyset) \leq_e J_e(C)$. It follows that $J_e(B_1 \oplus 
C)=K_{B_1 \oplus C}\oplus \left( K_{B_1 \oplus C}\right)^c\leq_e J_e(C)$, 
from which we get $J_e(B_1 \oplus C)\equiv_e J_e(C)$. 
\end{proof}

Let $\mathcal{QM}$ be the class of quasi-minimal sets $e$-below 
$J_e(\emptyset)$, plus the c.e.\ sets. 

\begin{corollary}\label{cor:qm}
$\mathcal{QM}$ is not a splitting class. 
\end{corollary}

\begin{proof}
We show that $\mathcal{QM}$ is not a splitting class by showing that there 
exist quasi-minimal sets $B_1, B_2$ for which there is no set $C 
\in \mathcal{QM}$ such that both $B_1 \oplus C$ and $B_2 \oplus C$ lie 
outside of $\mathcal{QM}$. Thus $\mathcal{QM}$ is not a splitting class since 
Definition~\ref{def:splittingclass} is falsified by taking for instance $a$ 
to be the set $A= \emptyset$ and the finite family $\mathcal{B}=\{b_1, 
b_2\}$ to be the family consisting of the above two sets $\{B_1, B_2\}$, for 
which $B_1\nleq_e A, B_2 \nleq_e A$, but there is no $C\in \mathcal{QM}$ such 
that both $B_1 \oplus C$ and $B_2 \oplus C$ lie outside of $\mathcal{QM}$. To 
show this, take a nontrivial $\mathcal{K}$-pair $\{B_1, B_2\}$. 
By Lemma~\ref{lem:Kprinc}(5), both $B_1$ and $B_2$ are quasi-minimal. Towards 
a contradiction, assume $C$ to be a quasi-minimal set such that both $B_1 
\oplus C$ and $B_2 \oplus C$ are not quasi-minimal. Then $B_1$ and $B_2$ are 
not $e$-reducible to $C$, so $\{B_1 \oplus C$, $B_2 \oplus C\}$ is a nontrivial 
$\mathcal{K}$-pair relative to $C$. Let $D_1$, $D_2$ be two total non c.e.\ 
sets $e$-reducible to $B_1 \oplus C$ and $B_2 \oplus C$ respectively. Then 
$D_1, D_2$ are not $e$-reducible to $C$ (as $C$ is quasi-minimal), so by 
Lemma~\ref{lem:Kprinc}(2) they are a nontrivial $\mathcal{K}$-pair relative 
to $C$. From here (again by Lemma~\ref{lem:Kprinc}(3) and by the totality of 
$D_2$), we have $D_1 \leq_e C \oplus D_2^c \leq_e C \oplus D_2$, which, 
together with $D_1 \leq_e C \oplus D_1$, by Lemma~\ref{lem:Kprinc}(4) implies 
$D_1 \leq_e C$, contradicting that $C$ is quasi-minimal. 
\end{proof}

Contrary to what happens in the Turing degrees, where the sets of low Turing 
degree form a splitting class (see \cite{Kuyper-Medvedev}) we have the 
following. 

\begin{corollary}\label{cor:low}
The sets with low $e$-degree do not form a splitting class in the 
$e$-degrees. 
\end{corollary}

\begin{proof}
The argument is similar to the one given in the previous proof. Take a 
nontrivial $\mathcal{K}$-pair $\{B_1, B_2\}$ below $K^c$ 
(see~\cite{Kalimullin}). By Corollary~\ref{cor:Kprinc} both $B_1$ and $B_2$ 
are $e$-low. Assume $C$ to be an $e$-low set such that both $B_1 \oplus C$ 
and $B_2 \oplus C$ are not $e$-low. Then by Corollary~\ref{cor:Kprinc} we 
have $J_e(B_1 \oplus C)\equiv_e J_e(C)$, which implies that $B_1 \oplus C$ is 
low, a contradiction. Therefore the pair $A=\emptyset$ and 
$\mathcal{B}=\{B_1,B_2\}$ witnesses that the $e$-low sets do not form a 
splitting class.   
\end{proof}

\begin{corollary}\label{cor:arithmetical}
The class of arithmetical sets $A\ngeq_e K^c$ (i.e. the arithmetical sets 
with $e$-degree not above the first $e$-jump) is not a splitting class. 
\end{corollary}

\begin{proof}
Take a $\mathcal{K}$-pair $\{B_1, B_2\}$ with $B_1, B_2$ in the class. Let 
$C$ be in the class (in particular $K^c \nleq_e C$), and assume that 
$B_1\oplus C$ and $B_2 \oplus C$ are not in the class. Then since $K^c\leq_e 
B_1\oplus C$ and $K^c\leq_e B_2\oplus C$, by Lemma~\ref{lem:Kprinc}(4) we 
have $K^c\leq_e C$, contradiction. Again, the pair $A=\emptyset$ and 
$\mathcal{B}=\{B_1,B_2\}$ witnesses that the arithmetical sets $\ngeq_e 
K^c$ do not form a splitting class. 
\end{proof}

In view of these negative results we ask:

\begin{question}
Are there natural splitting classes in the $e$-degrees?
\end{question}

\section{Evaluating $\Th(\mathscr{M}_{D})$ and 
$\Th(\mathscr{M}_{D,w})$}\label{sec:total} 

Having shown that there exist initial segments of the Dyment lattice 
$\mathscr{M}_D$ and of the Dyment-Muchnik lattice $\mathscr{M}_{D,w}$ whose 
logics coincide with the theorems of the intuitionistic propositional logic, 
we may also infer: 

\begin{theorem} 
$\Th(\mathscr{M}_{D})=\Th(\mathscr{M}_{D,w})=\JAN$. 
\end{theorem} 

\begin{proof} 
If $\mathscr{A}$ is a Brouwer algebra, and $\mathscr{B}$ is the Brouwer 
algebra obtained by adding to $\mathscr{A}$ an element on top of all the 
elements of $\mathscr{A}$, then for every \emph{positive} propositional 
formula $\sigma$ (i.e. a formula not containing $\neg$) we have that if 
$\sigma$ is true in $\mathscr{B}$ then is true in $\mathscr{A}$ as well. This 
follows from the following observation: assume that $\sigma$ is true in 
$\mathscr{B}$, and consider any evaluation of $\sigma$ in $\mathscr{A}$. But 
this is also an evaluation of $\sigma$ in $\mathscr{B}$, and gives $0$ in 
$\mathscr{B}$. But then it gives $0$ also in $\mathscr{A}$ as the only 
operation that can take elements of $\mathscr{A}$ (viewed as elements of 
$\mathscr{B}$) to $1$ in $\mathscr{B}$ is $\neg$. It follows that the 
positive propositional formulas that are valid in $\mathscr{M}_{D}$ are 
exactly the positive formulas in $\IPC$. On the other hand $\JAN \subseteq 
\Th(\mathscr{M}_{D})$, as it is easy to see that the greatest element of 
$\mathscr{M}_{D}$ is join-irreducible and this implies that $\neg \alpha \vee 
\neg \neg \alpha$ is valid in $\mathscr{M}_D$, as for every Dyment degree 
$\mathbf{A}$ we have that $\neg \mathbf{A}=\mathbf{1}_{D}$ if $\mathbf{A} \ne 
\mathbf{1}_{D}$, otherwise $\neg \mathbf{A}=\mathbf{0}_{D}$. The result 
follows by maximality of $\JAN$, among the intermediate logics whose positive 
fragment coincides with that of $\IPC$: see \cite{Jankov:Weak}. 

By a similar argument, having shown that there exists an initial segment of 
the Dyment-Muchnik lattice $\mathscr{M}_{D,w}$ whose logic coincides with the 
theorems of the intuitionistic propositional logic, we also conclude that 
$\Th(\mathscr{M}_{D,w})=\JAN$. 
\end{proof}

The identity $\Th(\mathscr{M}_{D})=\JAN$ also follows from the following 
theorem (as transfers to the Dyment lattice a result already proved in 
\cite{Sorbi:Embedding} for the Medvedev lattice), characterizing the finite 
Brouwer algebras that are embeddable in the Dyment lattice. 
\begin{theorem}\label{thm:another} The finite Brouwer algebras that 
are embeddable into $\Dym$ coincide with the finite Brouwer algebras in which 
$0$ is meet-irreducible and $1$ is join-irreducible. 
\end{theorem} 
\begin{proof} (Sketch) In one direction (i.e., to show that every finite 
Brouwer algebra embeddable in $\Dym$ has meet-irreducible $0$ and 
join-irreducible $1$) the claim is trivial since $\Dym$ has meet-irreducible 
$0$ and join-irreducible $1$. For the other direction, we observe that the 
proof given in \cite{Sorbi:Embedding} for the Medvedev lattice makes use of 
Medvedev mass problems of the form $\mathcal{B}_f=\{g: g\nleq_T f\}$ and is 
based on the fact that the degrees of these mass problems are 
join-irreducible and form a canonical set. But this is true also for the 
Dyment mass problems of the form $\mathcal{B}_\phi=\{\phi: \psi \nleq_e 
\phi\}$. Indeed, checking join-irreducibility in this case is 
straightforward, and these mass problems are Dyment–Muchnik mass problems, so 
Lemma~\ref{lem:can1} applies to them. The remainder of the proof in 
\cite{Sorbi:Embedding} is purely algebraic and carries over unchanged to the 
setting of the Dyment lattice. 
\end{proof}


\end{document}